\documentclass[12pt]{amsart}

\usepackage{fullpage}
\usepackage{amsmath,amsfonts,amssymb}
\usepackage[all,2cell]{xy}
\UseTwocells
\usepackage[usenames,dvipsnames]{color}
\usepackage{pdfcolmk}

\usepackage{pb-diagram,pb-xy}

\usepackage{hyperref}

\newcommand{\mathlarger}{}

\numberwithin{equation}{subsection}

\theoremstyle{theorem}
	\newtheorem{theorem}[equation]{Theorem}
	\newtheorem{lemma}[equation]{Lemma}
	\newtheorem{corollary}[equation]{Corollary}
	\newtheorem{proposition}[equation]{Proposition}

	\newtheorem*{theorem*}{Theorem}

\theoremstyle{definition}
	\newtheorem{definition}[equation]{Definition}
	\newtheorem{example}[equation]{Example}

\theoremstyle{remark}
	
	\newtheorem*{remark*}{Remark}

\newcommand{\norm}[1]{\lVert#1\rVert}

\newcommand{\st}{\mid}
\newcommand{\eps}{\varepsilon}
\newcommand{\To}{\Rightarrow}

\newcommand{\R}{\mathbb{R}}

\newcommand{\Z}{\mathbb{Z}}
\newcommand{\F}{\mathbb{F}}

\newcommand{\D}{\mathfrak{D}}

\newcommand{\met}{\operatorname{\rm d}}
\newcommand{\meto}{\met^\omega}
\newcommand{\metO}{\met^\Omega}

\newcommand{\Inv}{\operatorname{Inv}}
\newcommand{\pre}[1]{{#1}^{-1}}

\newcommand{\haus}{d_{\mathrm{H}}}
\newcommand{\antihaus}{\ell_{\mathrm{H}}}

\newcommand{\Iota}{{\rm I}}
\newcommand{\Kappa}{{\rm K}}
\newcommand{\Mu}{{\rm M}}

\newcommand{\cat}[1]{\mathbf{#1}}

\newcommand{\vect}[1]{{\mathbf{#1}}}

\newcommand{\Vect}{\cat{Vect}}
\newcommand{\Subs}{\cat{Subsets}}
\newcommand{\Trans}{\cat{Trans}}

\newcommand{\intr}{\operatorname{int}}

\newcommand{\tensor}{\otimes}
\newcommand{\Id}{\Iota}

 \newenvironment{vlist}%
 { \begin{list}%
         {$\bullet$}%
         {\setlength{\labelwidth}{20pt}%
          \setlength{\leftmargin}{25pt}%
          \setlength{\topsep}{0pt}
          \setlength{\itemsep}{1.5ex}
          \setlength{\parsep}{0pt}}}%
 { \end{list} }

\begin{document}

\title{Metrics for generalized persistence modules}
\author{Peter Bubenik}
\address[Peter Bubenik]{Department of Mathematics, Cleveland State University}
\email{p.bubenik@csuohio.edu}

\author{Vin de Silva}
\address[Vin de Silva]{Department of Mathematics, Pomona College}
\email{Vin.deSilva@pomona.edu}

\author{Jonathan Scott}
\address[Jonathan Scott]{Department of Mathematics, Cleveland State University}
\email{j.a.scott3@csuohio.edu}
\date{\today}

\begin{abstract}
We consider the question of defining interleaving metrics on generalized persistence modules over arbitrary preordered sets. Our constructions are functorial, which implies a form of stability for these metrics. We describe a large class of examples, inverse-image persistence modules, which occur whenever a topological space is mapped to a metric space. Several standard theories of persistence and their stability can be described in this framework. This includes the classical case of sublevelset persistent homology.
We introduce a distinction between `soft' and `hard' stability theorems.
While our treatment is direct and elementary, the approach can be explained abstractly in terms of monoidal functors.
\end{abstract}

\maketitle


\section*{ Introduction}

Topological persistence, most commonly seen in the form of persistent homology, is perhaps the core technology at the heart of topological data analysis. Given a finite data set sampled from an object of interest, one seeks to recover topological invariants of the object by suitable computations on the sampled data.
The stumbling block is that the classical invariants of algebraic topology are extremely sensitive to local fluctuations in the data, but sampling irregularities are unavoidable when working with real scientific data.

Persistence solves this problem by replacing individual topological invariants with systems of invariants. Instead of calculating the homology of a Vietoris--Rips complex built from the data at one fixed scale, persistence calculates the homology at all scales simultaneously. Quantities that are unstable at a fixed scale are seen to be stable when viewed across the full range of scales.

Early versions of this idea are seen in the work of Frosini~\cite{Frosini_1991} on size functions, and later Robins~\cite{Robins_1999} in studying fractal attractors of dynamical systems.

The big computational breakthrough came when Edelsbrunner, Letscher and Zomorodian~\cite{Edelsbrunner_L_Z_2000} introduced an algorithm that takes an increasing nested sequence of simplicial complexes (built from the data at increasing values of a scale parameter) and returns a compact descriptor of its homology: the \emph{persistence diagram} or \emph{barcode}. In subsequent work, Zomorodian and Carlsson~
\cite{Zomorodian_Carlsson_2005} expanded the reach of the algorithm by interpreting it in terms of commutative algebra; and Cohen-Steiner, Edelsbrunner and Harer~\cite{CohenSteiner_E_H_2007} gave the first proof of the crucial result that this persistence diagram is numerically stable.

Since then, there have been several variations on the persistence theme. These include multidimensional persistence, for multiparameter nested families~\cite{carlssonZomorodian:multidimP,Lesnick_2011}; zigzag persistence, for families that are not nested \cite{Carlsson_deSilva_2010}; the closely related theories of extended persistence~\cite{CS_E_H_2008} and levelset persistence~\cite{Carlsson_dS_M_2009}; angular persistence for spaces equipped with a circle-valued map~\cite{Dey_Burghelea_2011}; image, kernel, and cokernel persistence~\cite{csehm:kernels}; and so on.

An important insight was the realization that persistence is not so much about topology as it is about algebra. This was first pointed out in~\cite{Zomorodian_Carlsson_2005}, and has been exploited and developed by several authors since then, most notably in~\cite{Chazal_CS_G_G_O_2009} and~\cite{Lesnick_2011}.

For instance, consider a nested family of simplicial complexes:
\[
\begin{diagram}
\node{X_0}
	\arrow{e}
\node{X_1}
	\arrow{e}
\node{\cdots}
	\arrow{e}
\node{X_{n-1}}
\end{diagram}
\tag{$*$}
\]
Each arrow represents the inclusion of one complex in the next. Now let $H$ denote the operation of taking homology in a given dimension, with coefficients in a field~$\F$. This leads to a diagram of {vector spaces} and {linear maps}:
\[
\begin{diagram}
\node{H(X_0)}
	\arrow{e}
\node{H(X_1)}
	\arrow{e}
\node{\cdots}
	\arrow{e}
\node{H(X_{n-1})}
\end{diagram}
\tag{$H*$}
\]
The key observation of~\cite{Zomorodian_Carlsson_2005} is that the persistence diagram\footnote{%
The word `diagram' is used in two senses here, for unavoidable historical reasons.
}
produced by~\cite{Edelsbrunner_L_Z_2000} depends only on the algebraic structure carried by this diagram: it is a description of its isomorphism type, within the category of all such diagrams of vector spaces.
In this way, the theory of persistence is emancipated from topology.

We wish to encourage a further emancipation. Diagram~($*$) produces diagram~($H*$) because the operation $H$ is a \emph{functor} in the sense of category theory: it turns simplicial complexes into vector spaces, {and} it converts simplicial maps into linear maps in a way that respects identity maps and composition.

Our purpose in this paper is to develop the categorical point of view, with a special emphasis on metrics and stability.
This continues the program started in~\cite{bubenikScott:1}.
We work with generalized persistence modules defined over an arbitrary poset or preordered set, taking values in an arbitrary category. Two such modules may be compared by finding an `interleaving' between them. The additional data of a `sublinear projection' or a `superlinear family' quantifies the comparison and allows us to define a metric on generalized persistence modules.
The classical stability theorem can be seen as splitting into a `soft' categorical statement and a `hard' invariant theory statement; we show that soft stability is a very widespread phenomenon for the metrics we have defined.

\begin{remark*}
We believe that the language of category theory is natural for studying persistence modules, and simple to use once it has become familiar.
To help ensure that this point of view is useful to the reader, rather than a new burden, we have made every effort to keep the discussion as concrete and explicit as possible.
\end{remark*}

\medskip
{\bf Outline.}
In Section~\ref{sec:gpm}, we define generalized persistence modules and give a preview of the main theorems.
Section~\ref{sec:interleaving} is about the theory of interleaving metrics on categories of persistence modules.
We describe a large class of examples in Section~\ref{sec:map-to-metric}, which includes many known theories of persistence.
Finally, in Section~\ref{sec:monoids} we increase the level of abstraction and explain how our results connect with the theory of adjoint functors on monoidal categories. This leads to some suggestions on how to handle multidimensional persistence.

\section{Persistence Modules as Functors}
\label{sec:gpm}

\subsection{Categories}
We recall the definitions, informally. For more details see a standard reference such as \cite{maclane:book}. For a longer discussion in the context of persistent homology, see~\cite{bubenikScott:1}.

\begin{vlist}
\item
A \emph{category} consists of a collection of objects and collections of {morphisms} (or `arrows') between each pair of objects. There is a composition operation on morphisms which is associative; and there is an identity morphism from each object to itself.
 
\item
A \emph{functor} between categories is a map which takes objects to objects and morphisms to morphisms and is compatible with the structure of a category.
\end{vlist}

For instance, simplicial homology~$H$ in a given dimension is a functor from the category $\cat{Simp}$, of simplicial complexes and simplicial maps, to the category $\Vect_\F$, of vector spaces and linear maps over~$\F$.
We can use $H$ to convert any diagram in $\cat{Simp}$ to a diagram in~$\Vect_\F$, as we saw with ($*$) and ($H*$) in the introduction.

There is one more category lurking in that scenario. Consider the category $\cat{n}$ whose objects are the set $\{0, 1, \dots, n-1\}$, and with a unique morphism $j \to k$ for each $j \leq k$ and no morphism if $j > k$. We can visualize this category using the following diagram:
\[
\begin{diagram}
\node{0}
	\arrow{e}
\node{1}
	\arrow{e}
\node{\cdots}
	\arrow{e}
\node{n-1}
\end{diagram}
\]
Only the `generating' morphisms $j \to j+1$ are drawn here. The identity morphisms are implicitly understood to be there, and other morphisms $j \to k$ can be formed as compositions of generating morphisms.

The diagram (*) above may be thought of as a functor $X : \cat{n} \to \cat{Simp}$. This functor takes each object~$j$ to the corresponding simplicial complex~$X_j$, and each morphism $j \to k$ to the inclusion map $X_j \to X_k$ (which exists because $X_j \subseteq X_k$ whenever $j \leq k$).

We call $\cat{n}$ the \emph{indexing category} for~($*$).

Going further, the diagram ($H*$) also has $\cat{n}$ as its indexing category, and may be thought of as the composite functor $HX$:
\[
\begin{diagram}
\node{\cat{n}}
	\arrow{e,t}{X}
\node{\cat{Simp}}
	\arrow{e,t}{H}
\node{\Vect_{\F}}
\end{diagram}
\]
Thus, beginning with a category $\cat{n}$ which specifies a particular `shape' of diagram, other diagrams of the same shape can be thought of as functors defined on~$\cat{n}$.

We can define maps between diagrams. Consider two vector-space valued diagrams
$V, W : \cat{n} \to \Vect_{\F}$. A map $\varphi: V \to W$ is defined to be a collection of linear maps $(\varphi_j : V_j \to W_j)$ such that all squares of the following diagram commute:
\[
\begin{diagram}
\dgARROWLENGTH=1.5em
\node{V_0}
	\arrow{e}
	\arrow{s,l}{\varphi_0}
\node{V_1}
	\arrow{e}
	\arrow{s,l}{\varphi_1}
\node{\dots}
	\arrow{e}
\node{V_{n-1}}
	\arrow{s,r}{\varphi_{n-1}}
\\
\node{W_0}
	\arrow{e}
\node{W_1}
	\arrow{e}
\node{\dots}
	\arrow{e}
\node{W_{n-1}}
\end{diagram}
\]
This too can be interpreted in category theory.
\begin{vlist}
\item
A \emph{natural transformation} $\varphi: F \To G$ between two functors $F,G: \cat{P} \to \cat{D}$ is a collection $(\varphi_x)$ of morphisms in~$\cat{D}$. For each object $x$ of~$\cat{P}$ we have $\varphi_x: F(x) \to G(x)$, and we require that the diagram
\begin{equation}
\begin{diagram}
\dgARROWLENGTH=1.5em
\node{F(x)}
	\arrow{e,t}{F(\alpha)}
	\arrow{s,l}{\varphi_x}
\node{F(y)}
	\arrow{s,r}{\varphi_y}
\\
\node{G(x)}
	\arrow{e,t}{G(\alpha)}
\node{G(y)}
\end{diagram}
\label{eq:natural}
\end{equation}
commute for every $\cat{P}$-morphism $\alpha: x \to y$.
\end{vlist}
Then, a map between two diagrams $F,G$ is precisely a natural transformation $F \To G$.

\begin{example}
\label{ex:natural}
Suppose we are given two nested families of simplicial complexes
\[
X_0 \subseteq X_1 \subseteq X_2 \subseteq X_3
\quad
\text{and}
\quad
Y_0 \subseteq Y_1 \subseteq Y_2 \subseteq Y_3.
\]
These can be thought of as functors $X, Y: \cat{4} \to \cat{Simp}$. Suppose $f_3: X_3 \to Y_3$ is a simplicial map which restricts to simplicial maps $f_i: X_i \to Y_i$ for all~$i$. Then $f = (f_i)$ is a natural transformation $X \To Y$.
The commutative diagram
\[
\begin{diagram}
\dgARROWLENGTH=3em
\node{X_0}
	\arrow{e}
	\arrow{s,l}{f_0}
\node{X_1}
	\arrow{e}
	\arrow{s,l}{f_1}
\node{X_2}
	\arrow{e}
	\arrow{s,l}{f_2}
\node{X_3}
	\arrow{s,l}{f_3}
\\
\node{Y_0}
	\arrow{e}
\node{Y_1}
	\arrow{e}
\node{Y_2}
	\arrow{e}
\node{Y_3}
\end{diagram}
\]
shows the maps and contains all the required commutative squares.
\end{example}

There are many kinds of composition involving natural transformations. Here are the ones we use.

\begin{vlist}
\item
If $\varphi: F \To G$ and $\psi: G \To H$ where $F,G,H$ are functors $\cat{P} \to \cat{D}$ then their `vertical composition' is the natural transformation $\psi \varphi: F \To H$ defined by $(\psi \varphi)_x = \psi_x \phi_x$.
\end{vlist}

A natural transformation $\varphi: F \To G$, where $F, G: \cat{P} \to \cat{D}$, may be `feathered' on either side by a functor. 

\begin{vlist}
\item
If $H: \cat{D} \to \cat{D}'$ then $H \varphi: HF \To HG$ is the natural transformation defined by $(H\varphi)_x = H(\varphi_x)$, the morphism resulting from applying the functor $H$ to the morphism $\varphi_x$.

\item
If $K: \cat{P}' \to \cat{P}$ then $\varphi K: FK \To GK$ is the natural transformation defined by $(\varphi K)_{x'} = \varphi_{K(x')}$.

\end{vlist}

%
We illustrate feathering using the natural transformation $f$ of Example~\ref{ex:natural}:

\begin{example}
Let $H: \cat{Simp} \to \cat{Vect}_\F$ be a homology functor. This leads to functors $HX, HY: \cat{4} \to \cat{Vect}_\F$. The commutative diagram
\[
\begin{diagram}
\dgARROWLENGTH=2.25em
\node{H(X_0)}
	\arrow{e}
	\arrow{s,l}{H(f_0)}
\node{H(X_1)}
	\arrow{e}
	\arrow{s,l}{H(f_1)}
\node{H(X_2)}
	\arrow{e}
	\arrow{s,l}{H(f_2)}
\node{H(X_3)}
	\arrow{s,l}{H(f_3)}
\\
\node{H(Y_0)}
	\arrow{e}
\node{H(Y_1)}
	\arrow{e}
\node{H(Y_2)}
	\arrow{e}
\node{H(Y_3)}
\end{diagram}
\]
represents the natural transformation $Hf: HX \To HY$.
\end{example}

\begin{example}
Pick any functor $K: \cat{2} \to \cat{4}$; for example, the index-doubling functor defined by setting $K(0) = 0$ and $K(1) = 2$. This leads to functors $XK, YK: \cat{2} \to \cat{Simp}$.
The commutative diagram
\[
\begin{diagram}
\dgARROWLENGTH=3em
\node{X_0}
	\arrow[2]{e}
	\arrow{s,l}{f_0}
\node[2]{X_2}
	\arrow{s,l}{f_2}
\\
\node{Y_0}
	\arrow[2]{e}
\node[2]{Y_2}
\end{diagram}
\]
represents the natural transformation $fK: XK \To YK$.
\end{example}

The name `feathering' is suggested by the look of the relevant parts of this diagram:
\[
\xymatrix{
\cat{P}' 
	\ar[r]^K
&
\cat{P}
	\ar@/^1pc/[r]^F
	\ar@/_1pc/[r]_G
&
\cat{D}
	\ar[r]^H
&
\cat{D}'
}
\]
Double-feathering $H \varphi K: HFK \To HGK$ may be defined similarly. 
We use parentheses to disambiguate more complicated composites of natural transformations and functors. The relations $H(\psi \varphi) = (H \psi) (H \varphi)$ and $(\psi \varphi)K = (\psi K) (\varphi K)$ are used several times.

\subsection{Generalized persistence modules}

The term `persistence module' was introduced by Zomorodian and Carlsson in~\cite{Zomorodian_Carlsson_2005}, to mean a collection of vector spaces $V_i$ and linear maps $V_i \to V_{i+1}$, where the indices are the natural numbers. This data can be represented as graded module over the polynomial ring $\F[t]$, which enables the use of classical module theory to define the persistence diagram and construct it algorithmically.
Subsequently, Chazal et~al.~\cite{Chazal_CS_G_G_O_2009} considered persistence modules indexed over the real numbers, using analytic arguments to construct the persistence diagram and establish its stability.

These persistence modules are exactly the same thing as functors $\cat{N} \to \Vect_\F$ or $\cat{R} \to \Vect_\F$. Here $\cat{N}$  (respectively~$\cat{R})$ is the category whose objects are the natural numbers (respectively the real numbers) and with a unique morphism $s \to t$ whenever $s \leq t$.

For the present paper\footnote{%
In the earlier work~\cite{bubenikScott:1} the target category is arbitrary and the indexing category is~$\cat{R}$.
}
we generalize the notion of persistence module as follows:

\begin{vlist}
\item
The indexing category ($\cat{n}$, $\cat{N}$ or $\cat{R}$ above) may be any preordered set~$\cat{P}$ (Section~\ref{subsec:proset}).

\item
The target category ($\Vect_\F$ above) may be any category~$\cat{D}$.

\item
Functors $\cat{P} \to \cat{D}$ are called \emph{generalized persistence modules, in~$\cat{D}$, over~$\cat{P}$}.
\end{vlist}

The collection of functors $\cat{P} \to \cat{D}$ is itself a category, its morphisms being the natural transformations between functors. 
This category is called $\cat{D}^{\cat{P}}$.
(This is analogous to writing $Y^X$ for the set of functions $X \to Y$.)

The word `module' has a strong association with abelian categories~\cite{maclane:book,bubenikScott:1} but in this paper we use the word in a loose sense.

The value of working in this level of generality may be seen in the following examples.

\begin{example}[the sublevelset module]
\label{ex:gpm-1}
Let $f: X \to \R$ be a function on a topological space. This defines a functor $F \in \cat{Top}^\cat{R}$ where $\cat{Top}$ is the category of topological spaces. Specifically, for every real number~$t$ we define $F(t)$ to be the subspace $\pre{f}(-\infty,t] \subseteq X$ and for every pair $s \leq t$ the morphism $F(s) \to F(t)$ is defined to be the inclusion map.
\end{example}

From the sublevelset module we reach other well-known constructions.

\begin{example}
\label{ex:gpm-2}
The sublevelset persistent homology of~$f$ is obtained as a composite
\[
\begin{diagram}
\node{\cat{R}}
	\arrow{e,t}{F}
\node{\cat{Top}}
	\arrow{e,t}{H}
\node{\Vect_\F}
\end{diagram}
\]
where $H$ is homology with coefficients in the field~$\F$.
\end{example}

\begin{example}
\label{ex:gpm-3}
The merge tree of~$f$ is obtained as a composite
\[
\begin{diagram}
\node{\cat{R}}
	\arrow{e,t}{F}
\node{\cat{Top}}
	\arrow{e,t}{\pi_0}
\node{\cat{Set}}
\end{diagram}
\]
where $\pi_0$ returns the set of connected components of a topological space. Explicitly: for each $t \in \R$ the composite functor returns the set of connected components of $\pre{f}(-\infty,t]$; and for $s \leq t$ it returns the map by which the components of $\pre{f}(-\infty,s]$ include into the components of $\pre{f}(-\infty,t]$.
One can use this information to build a topological graph that corresponds to the traditional notion of merge tree, of the kind studied by Morozov et~al.~\cite{Morozov_B_W_2013}.
\end{example}

Using feathering, the stability of Examples \ref{ex:gpm-2} and~\ref{ex:gpm-3} follows automatically from the stability of Example~\ref{ex:gpm-1}. We will discuss this in detail later on.

\subsection{Stability}
\label{subsec:stability}

The stability theorem of Cohen-Steiner, Edelsbrunner and Harer~\cite{CohenSteiner_E_H_2007} is of such central importance that any worthwhile theoretical language for persistence ought to be able to describe it.

Our approach requires additional structure on the indexing category, which enables us to define an \emph{interleaving distance} between generalized persistence modules.
We have two proposals for what this additional structure should be, both of which have the desired effect:

\begin{vlist}
\item
$\cat{P}$ is a preordered set equipped with a \emph{sublinear projection} on translations.

\item
$\cat{P}$ is a preordered set equipped with a \emph{superlinear family} of translations.
\end{vlist}

We will define these terms in Sections \ref{sec:sublinear} and~\ref{sec:superlinear}, respectively; and in Section~\ref{sec:monoids} we reinterpret them in terms of monoidal categories.
Most commonly we use option three:

\begin{vlist}
\item
$\cat{P}$ is a preordered set with a metric or Lawvere metric.
\end{vlist}

In Section~\ref{sec:lawvere}, we explain how to obtain a sublinear
projection from such a metric.

Here are the main results.
The stability and the inverse-image stability theorems are reformulations of results of~\cite{CohenSteiner_E_H_2007,Chazal_CS_G_G_O_2009} from classical sublevelset persistence, in our more general context but using essentially the same arguments. 

\begin{theorem*}[interleaving distance, \ref{thm:inter-sub} and~\ref{thm:inter-super}]
\label{int-thm:inter}
Let $\cat{P}$ be a preordered set and let $\cat{D}$ be an arbitrary category. Suppose $\cat{P}$ is equipped with a sublinear projection or superlinear family. The extra structure induces an distance function $\met(F,G)$ between functors $F, G: \cat{P} \to \cat{D}$. This \emph{`interleaving distance'} is an extended pseudometric on~$\cat{D}^{\cat{P}}$.
\end{theorem*}

\begin{theorem*}[stability, \ref{thm:sub-stability} and~\ref{thm:super-stability}]
\label{int-thm:stab}
Let $\cat{P}$ be a preordered set with a sublinear projection or superlinear family. Let $H: \cat{D} \to \cat{E}$ be a functor between arbitrary categories $\cat{D}, \cat{E}$. Then for any two functors $F, G: \cat{P} \to \cat{D}$ we have
\[
\met(HF,HG) \leq \met(F,G).
\]
That is, the map $H^\cat{P}: \cat{D}^{\cat{P}} \to \cat{E}^{\cat{P}}$, defined by post-composing with~$H$, is $1$-Lipschitz.
\end{theorem*}

The \emph{inverse-image construction} in Section~\ref{sec:map-to-metric} is a very common way to construct persistence modules. Let $\cat{P}$ be a poset of subsets of a metric space~$Y$. Let $X$ be a topological space equipped with a function $f: X \to Y$ (which need not be continuous). Then the functor $F$ defined from~$f$ by
\[
\begin{diagram}
\node{\cat{P}}
	\arrow{e,t}{F}
\node{\cat{Top}}
\\
\node{A}
	\arrow{e,T}
\node{\pre{f}(A)}
\end{diagram}
\]
is a persistence module in the category of topological spaces.

There is a natural way to define a Lawvere metric on~$\cat{P}$. If $\cat{P}$ satisfies a certain closure condition then the following theorem holds:

\begin{theorem*}[inverse-image stability, \ref{thm:inv-stab}]
\label{int-thm:inv-stab}
Let $X, Y$ and $\cat{P}$ be as above.
Let $f, g: X \to Y$ be functions and let $F, G \in \cat{Top}^\cat{P}$ be the corresponding persistence modules. Then
\[
\met(F,G) \leq
\met_\infty(f,g) := \sup_{x \in X} d_Y(f(x), g(x))
\]
provided that $\cat{P}$ `has enough translations' (Definition~\ref{def:enough}).
\end{theorem*}

By way of illustration, here is the most familiar instance~\cite{CohenSteiner_E_H_2007,Chazal_CS_G_G_O_2009} of these results.

\begin{example}
Let $f, g: X \to \R$ be real-valued functions on a topological space $X$. Then, for any homology functor~$H$, the sublevelset persistent homology modules
\begin{align*}
HF &= 
\left( H( \pre{f}(-\infty, t]) \mid {t \in \R} \right)
\\
HG &= 
\left( H( \pre{g}(-\infty, t]) \mid {t \in \R} \right)
\end{align*}
have interleaving distance bounded as follows:
\[
\met(HF, HG) \leq  \norm{f -g}_\infty.
\]
\end{example}

The interleaving bound follows from the theorems stated above. 
A direct proof is not difficult to find, and indeed the theorems themselves have straightforward proofs.
The real value of these theorems is that they isolate the formal aspects of stability, and suggest what generalizations and adjustments are easily available.

%
In many instances, topological data analysis can be summarized by this workflow:
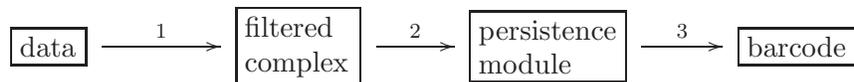
\[
\begin{diagram}
\small
\node{\boxed{\parbox{2em}{data}}}
	\arrow{e,t}{1}
\node{\boxed{\parbox{3.7em}{filtered\\complex}}}
	\arrow{e,t}{2}
\node{\boxed{\parbox{4.8em}{persistence\\module}}}
	\arrow{e,t}{3}
\node{\boxed{\parbox{3.5em}{barcode}}}
\end{diagram}
\]
%
Arrow~1 may be something like the Vietoris--Rips or \v{C}ech construction.
Arrows 2 and~3 are combined in the classical algorithm of~\cite{Edelsbrunner_L_Z_2000}.
The persistence stability theorem~\cite{CohenSteiner_E_H_2007} may be thought of as the assertion that this entire process is $1$-Lipschitz with respect to suitable metrics on data and on barcodes.

Of course, one can attempt this arrow-by-arrow. The algebraic stability theorem of~\cite{Chazal_CS_G_G_O_2009} asserts that arrow~3 by itself is 1-Lipschitz. In applications, it remains to verify that arrows 1 and~2 combine to give a 1-Lipschitz operation. In many situations this is easily done.

For generalized persistence modules, the workflow may look like this:
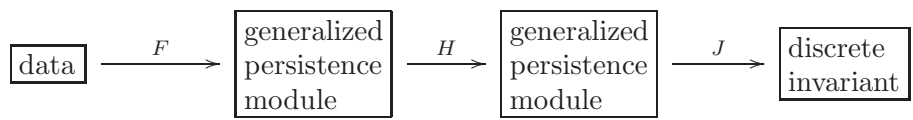
\[
\begin{diagram}
\small
\node{\boxed{\parbox{2em}{data}}}
	\arrow{e,t}{F}
\node{\boxed{\parbox{4.8em}{generalized\\persistence\\module}}}
	\arrow{e,t}{H}
\node{\boxed{\parbox{4.8em}{generalized\\persistence\\module}}}
	\arrow{e,t}{J}
\node{\boxed{\parbox{3.9em}{discrete\\invariant}}}
\end{diagram}
\]
Here $H$ denotes composition with a functor, as in Examples \ref{ex:gpm-2} and~\ref{ex:gpm-3}. Operations $F,J$ are not expected to be categorical.
We propose the following language:
\begin{vlist}
\item
\emph{Soft stability theorems} assert that $HF$ is Lipschitz (or at least uniformly continuous).

\item
\emph{Hard stability theorems} assert that $J$ is Lipschitz (or at least uniformly continuous).
\end{vlist}

Hard stability theorems (such as the algebraic stability theorem) are of enormous importance, but in this paper we focus on soft stability theorems and the high-level arguments that produce them.

\section{Interleaving metrics.}
\label{sec:interleaving}

We wish to develop proximity relationships between persistence modules. When are two persistence modules close to each other?

The simplest relationship is isomorphism.
Two persistence modules $F, G: \cat{P} \to \cat{D}$ are isomorphic if  there are natural transformations $\varphi: F \To G$ and $\psi: G \To F$ such that $\psi \varphi$ and $\varphi \psi$ equal the identity transformations on $F$ and~$G$, respectively.

For a metric between persistence modules, we need a notion of approximate isomorphism. The relevant notion is interleaving. In sections \ref{subsec:proset}--\ref{subsec:translations} we develop this idea qualitatively. In sections \ref{sec:sublinear}--\ref{sec:superlinear} we explain what additional data is required to define a metric: sublinear projections and superlinear families.

Our ideas are motivated by the special cases $\cat{P} = (\R, \leq)$ and $(\R^n, \leq)$ and the notion of interleaving articulated by~\cite{Chazal_CS_G_G_O_2009,Lesnick_2011} and others.

\subsection{Preordered sets}
\label{subsec:proset}

We will develop a theory of interleaving for persistence modules indexed by a poset. We continue to use language and concepts from category theory. In this realm it is more natural to work with \emph{preordered sets}.

A {preordered set} is a pair~$(P, \leq)$ where $P$ is a set and $\leq$ is a reflexive, transitive relation on~$P$. The collection of preordered sets forms a category $\cat{Proset}$. A morphism in $\cat{Proset}$ is a function $f: P \to Q$ that is \emph{monotone}, meaning that $p_1 \leq p_2$ implies $f(p_1) \leq f(p_2)$.

We can think of a preordered set as a category~$\cat{P}$ with an object for every element of~$P$ and a unique morphism $p_1 \to p_2$ whenever $p_1 \leq p_2$. In the other direction, a category which is \emph{small} and \emph{thin} may be interpreted as a preordered set;  `small' means that its objects form a set, and `thin' means that there is at most one morphism $p_1 \to p_2$ for any objects $p_1, p_2$.

A functor $\cat{P} \to \cat{Q}$ between two such categories  is uniquely specified by a monotone map $P \to Q$. In this way $\cat{Proset}$ may be interpreted as the category of small thin categories.

\begin{remark*}
A preordered set is a poset if $p_1 \leq p_2$ and $p_2 \leq p_1$ implies $p_1 = p_2$. Whether or not it is a poset, its equivalence classes under the relation
\[
p_1 \equiv p_2
\quad\Leftrightarrow\quad
p_1 \leq p_2
\;\;\text{and}\;\;
p_2 \leq p_1
\]
do form a poset.
In categorical language, a small thin category is a poset category if its only isomorphisms are the identity maps. It is not in the spirit of category theory to enforce this condition: a preordered set and its poset of equivalence classes give rise to categories $\cat{P}, \cat{\hat{P}}$ which are equivalent, and category theory does not seek to distinguish equivalent categories.
We note, for example, that the categories of persistence modules $\cat{D}^{\cat{P}}, \cat{D}^{\cat{\hat{P}}}$ are naturally isomorphic when $\cat{P}, \cat{\hat{P}}$ are equivalent; so it doesn't matter which one we work with.
\end{remark*}

The following lemma saves a lot of effort.

\begin{lemma}[The Thin Lemma]
\label{lem:thin}
In a thin category all diagrams commute, since there is at most one morphism between any two objects. One never needs to `check' that a diagram commutes; only that the arrows in the diagram exist in the first place.
\qed
\end{lemma}

Here is a sample use of the Thin Lemma.

\begin{corollary}
\label{cor:thin}
Let $F, G: \cat{C} \to \cat{Q}$ be functors into a thin category~$\cat{Q}$ (regarded as a preordered set). Then there is a natural transformation $\varphi: F \To G$ if and only if $F(x) \leq G(x)$ for all $x \in \cat{C}$. When it exists, the natural transformation is unique;
so $\cat{Q}^\cat{C}$ is thin.
\end{corollary}

\begin{proof}
Each map $\varphi_x$ of the natural transformation exists, and exists uniquely, if and only if the corresponding relation $F(x) \leq G(x)$ holds. Then the relations~\eqref{eq:natural} automatically hold by virtue of the Thin Lemma.
\end{proof}

\subsection{Translations and interleavings}
\label{subsec:translations}

A \emph{translation} on a preordered set $(P, \leq)$ is a function $\Gamma: P \to P$ which is monotone and which satisfies $x \leq \Gamma(x)$ for all $x \in P$.

In categorical language, using Corollary~\ref{cor:thin}, a translation is a functor $\Gamma: \cat{P} \to \cat{P}$ for which there exists a natural transformation $\eta_\Gamma: \Iota \To \Gamma$, where $\Iota$ is the identity functor.
The natural transformation is unique when it exists.

Let $\Trans_{\cat{P}}$ denote the set of translations of~$\cat{P}$. Then $\Trans_{\cat{P}}$ has the following structure:

\begin{vlist}
\item
It is a monoid (i.e.\ a semigroup with identity), with respect to composition.
\item
It is a preordered set, with respect to the relation
\begin{align*}
\Gamma \leq \Kappa
&\;\Leftrightarrow\;
\text{$\Gamma(x) \leq \Kappa(x)$ for all~$x$}
\\
&\;\Leftrightarrow\;
\text{there exists a natural transformation $\eta^\Gamma_\Kappa: \Gamma \To \Kappa$}
\end{align*}

\item
If $\cat{P}$ is a poset, then so is $\Trans_{\cat{P}}$.
\end{vlist}

The preorder is compatible with the monoid structure in the sense that
\begin{equation}
\label{eq:ordered-monoid}
\Gamma \leq \Kappa \text{ and } \Lambda \leq \Mu
\quad \Rightarrow \quad
\Gamma \Lambda \leq \Kappa \Mu.
\end{equation}
We can think of $\Trans_{\cat{P}}$ as a sort of `positive cone' in the monoid of all endomorphisms (i.e.\ monotone functions) $\cat{P} \to \cat{P}$.

The monoid $\cat{Trans_P}$ acts on $\cat{D}^\cat{P}$ by precomposition:
\[
\begin{diagram}
\dgARROWLENGTH=1.5em
\node{\cat{P}}
	\arrow{e,t}{\Gamma}
	\arrow{se,b}{F\Gamma}
\node{\cat{P}}
	\arrow{s,r}{F}
\\
\node[2]{\cat{D}}
\end{diagram}
\]
There is a canonical natural transformation $F\eta_\Gamma: F \To F\Gamma$, induced from the natural transformation $\eta_\Gamma: \Iota \To \Gamma$.
Specifically, the functor $F$ converts each relation $x \leq \Gamma(x)$ to a map $F(x) \to F\Gamma(x)$. 
We call this a \emph{shift map}.

\begin{example}
\label{ex:real-Omega}
Let $\cat{R} = (\R, \leq)$ and let $V$ be a persistence module over~$\cat{R}$.
For any $\eps \geq 0$, consider the translation $\Omega_\eps(t) = t + \eps$. This shifts the real line up by~$\eps$, therefore $V \Omega_\eps$ is obtained from $V$ by shifting all the information down by~$\eps$. 
Indeed
\[
[V \Omega_\eps](t) = V(t+ \eps)
\]
for all $t \in \R$.
The shift map $V \To V\Omega_\eps$ is given by the maps $V(t) \to V(t+\eps)$
provided by the persistence module structure, that is, by applying the functor~$V$ to the relation $t \leq t + \eps$.
\end{example}

\begin{remark*}
We can equally well define $V\Omega_\eps$ when $\eps < 0$, but in that case there is no shift map $V \To V\Omega_\eps$.
\end{remark*}

A translation~$\Gamma$ on a preordered set~$\cat{P}$ defines a relationship between persistence modules over~$\cat{P}$ called $\Gamma$-interleaving. Here is a non-symmetric version of this relationship. This kind of non-symmetric interleaving was considered by Lesnick~\cite{Lesnick_2011}, who defined $(J_1,J_2)$-interleavings for multidimensional persistence modules.

\begin{definition}
\label{def:interleaving}
Let $\cat{P}$ be a preordered set and let $\Gamma, \Kappa \in \Trans_{\cat{P}}$. Suppose $F, G \in \cat{D}^{\cat{P}}$ are persistence modules over~$\cat{P}$ in some category~$\cat{D}$.
Then a \emph{$(\Gamma,\Kappa)$-interleaving} between $F$ and~$G$ consists of  natural transformations $\varphi: F \To G \Gamma$ and $\psi: G \To F \Kappa$
\begin{equation*}
    \xymatrix{
    \cat{P} \ar[r]^{\Gamma} \ar[d]_F \ar@<1ex>@{}[dr]|*+{\stackrel{\mathlarger{\varphi}}{\To}} 
& \cat{P} \ar[d]_G \ar[r]^{\Kappa}
\ar@<1ex>@{}[dr]|*+{\stackrel{\mathlarger{\psi}}{\To}}  
& \cat{P} \ar[d]^F \\ 
   \cat{D} \ar@{=}[r] & \cat{D} \ar@{=}[r] & \cat{D}
    }
\end{equation*}
such that 
\begin{equation}
\label{eq:interleaving1}
(\psi \Gamma) \varphi = F \eta_{\Kappa\Gamma}
\quad \text{and} \quad
(\varphi \Kappa) \psi = G \eta_{\Gamma\Kappa}.
\end{equation}
We say that $F, G$ are \emph{$(\Gamma, \Kappa)$-interleaved} if there exists a $(\Gamma,\Kappa)$-interleaving between them.
\end{definition}

Explicitly:
$(\varphi,\psi)$ is a $(\Gamma,\Kappa)$-interleaving between $F, G$ if
the diagrams
\begin{equation}
\label{eq:nat}
\xymatrix@R=1em@C=0.5em{F(x) \ar[rr] \ar[dr]_{\varphi_x} & & F(y) \ar[dr]^{\varphi_y}\\
& G(\Gamma x) \ar[rr] & & G(\Gamma y)
}
\quad
\xymatrix@R=1em@C=0.5em{& F(\Kappa x) \ar[rr] & & F (\Kappa y)\\
G (x) \ar[ur]^{\psi_x} \ar[rr] & & G (y) \ar[ur]_{\psi_y}
}
\end{equation}
commute for all $x \leq y$ in~$\cat{P}$; and the diagrams
\begin{equation}
\label{eq:int}
\xymatrix@R=1em@C=0.5em{F (x) \ar[rr] \ar[dr]_{\varphi_x} & & F (\Kappa \Gamma x)\\
& G (\Gamma x) \ar[ur]_{\psi_{\Gamma x}}
}
\quad
\xymatrix@R=1em@C=0.5em{
& F (\Kappa x) \ar[dr]^{\varphi_{\Kappa x}}\\
G (x) \ar[rr] \ar[ur]^{\psi_x} & & G (\Gamma \Kappa x).
}
\end{equation}
commute for all $x \in \cat{P}$. 

\begin{definition}
A $(\Gamma, \Gamma)$-interleaving is called a $\Gamma$-interleaving, for short. We say that $F, G$ are $\Gamma$-interleaved if there exists a $\Gamma$-interleaving between them.
\end{definition}

\begin{example}
[continuing Example~\ref{ex:real-Omega}]
\label{ex:real-Omega2}
Let $V, W$ be persistence modules over~$\cat{R}$ and let $\eps \geq 0$.
A morphism $V \to W\Omega_\eps$ can be thought of as a morphism `of degree~$\eps$' from $V$ to~$W$, in that it provides a map $V(t) \to W({t+\eps})$ for every $t$. An $\Omega_\eps$-interleaving between $V, W$ is the same as an $\eps$-interleaving between $V, W$ in the classical sense~\cite{Chazal_CS_G_G_O_2009}.
\end{example}

\begin{example}
\label{ex:0-inter}
Two persistence modules are ${\Iota}$-interleaved---that is, $(\Iota, \Iota)$-interleaved---if and only if they are isomorphic.
\end{example}

When the target category is itself a preordered set, the Thin Lemma provides a simple test for the existence of an interleaving. We will use this test many times.

\begin{lemma}
\label{lem:poset}
Let $\cat{P}, \cat{Q}$ be preordered sets, let $F, G \in \cat{Q}^{\cat{P}}$, and let $\Gamma, \Kappa \in \Trans_{\cat{P}}$. Then $F,G$ are $(\Gamma,\Kappa)$-interleaved if and only if $F(x) \leq G \Gamma(x)$ and $G(x) \leq F \Kappa (x)$ for all $x \in \cat{P}$.
\end{lemma}

\begin{proof}
The existence of maps $\varphi_x: F(x) \to G \Gamma(x)$ and $\psi_x: G(x) \to F \Kappa  (x)$ is equivalent to the two sets of inequalities. All diagrams \eqref{eq:nat}, \eqref{eq:int}  commute by the Thin Lemma.
\end{proof}

Here are three formal properties of the interleaving relation: functoriality, monotonicity, and the triangle inequality. The proofs are all easy, even if they seem complicated.

\begin{proposition}[functoriality]
\label{prop:functorial}
Let $\Gamma$, $\Kappa$ be translations on a preordered set $\cat{P}$. Suppose $F, G : \cat{P} \to \cat{D}$ are persistence modules in~$\cat{D}$, and $H: \cat{D} \to \cat{E}$ is a functor. Then the statement
\[
\text{$F, G$ are $(\Gamma,\Kappa)$-interleaved}
\]
implies
\[
\text{$HF, HG$ are $(\Gamma,\Kappa)$-interleaved}.
\]
\end{proposition}

\begin{proof}
Let $\varphi : F \To G\Gamma$ and $\psi : G \To F\Kappa$ be the structure morphisms of the $(\Gamma,\Kappa)$-interleaving.  Since $H$ is a functor, $H \varphi : HF \To HG\Gamma$ and $H \psi : HG \To HF\Kappa$ are natural transformations.  

By functoriality, $(H\psi\Gamma)(H\varphi) = H((\psi\Gamma)\varphi) = HF\eta_{\Kappa\Gamma}$ and $(H\varphi\Kappa)(H\psi) = H((\varphi\Kappa)\psi) = HG\eta_{\Gamma\Kappa}$, and so $(H\varphi,H\psi)$ form the required interleaving.
\end{proof}

\begin{proposition}[monotonicity]
\label{prop:monotonicity}
Let $\Gamma_1, \Gamma_2, \Kappa_1, \Kappa_2$ be translations and let $F, G$ be persistence modules over a preordered set~$\cat{P}$. Suppose $\Gamma_1 \leq \Gamma_2$ and $\Kappa_1 \leq \Kappa_2$. Then the statement
\[
\text{$F$ and $G$ are $(\Gamma_1,\Kappa_1)$-interleaved}
\]
implies
\[
\text{$F$ and $G$ are $(\Gamma_2,\Kappa_2)$-interleaved}.
\]
\end{proposition}

\begin{proof}
The assertions $\Gamma_{1} \leq \Gamma_{2}$ and $\Kappa_{1} \leq \Kappa_{2}$ imply that there are natural transformations $\xi : \Gamma_{1} \To \Gamma_{2}$ and $\chi : \Kappa_{1} \To \Kappa_{2}$.  

Let $\varphi_{1} : F \To G \Gamma_{1}$ and $\psi_{1} : G \To F \Kappa_{1}$ form a $(\Gamma_{1}, \Gamma_{2})$-interleaving.  Let $\varphi_{2}$ and $\psi_{2}$ be the composites
\[
	F \stackrel{\varphi_{1}}{\Longrightarrow} G\Gamma_{1} \stackrel{G\xi}{\Longrightarrow} G\Gamma_{2}
	\quad
	\text{and}
	\quad
	G \stackrel{\psi_{1}}{\Longrightarrow} F\Kappa_{1} \stackrel{F\chi}{\Longrightarrow} F\Kappa_{2},
\]
respectively.

To show that $(\psi_{2} \Gamma_{2})\phi_{2} = F \eta_{\Kappa_{2}\Gamma_{2}}$, we consider the diagram
\[
	\xymatrix{
		F \ar@2[r]^{\varphi_{1}} \ar@2[dr]_{F\eta_{K_{1}\Gamma_{1}}}
			& G\Gamma_{1} \ar@2[r]^{G\xi} \ar@2[d]^{\psi_{1}\Gamma_{1}}
			& G\Gamma_{2} \ar@2[d]^{\psi_{1}\Gamma_{2}} \\
			& F\Kappa_{1}\Gamma_{1} \ar@2[r]_{F\Kappa_{1}\xi}  
			& F\Kappa_{1}\Gamma_{2} \ar@2[d]^{F\chi\Gamma_{2}} 	\\
			&  & F\Kappa_{2}\Gamma_{2}
	}
\]
of functors and natural transformations.  
The top-left triangle commutes because $(\varphi_{1},\psi_{1})$ is a $(\Gamma_{1},\Kappa_{1})$-interleaving.  The top-right square commutes because $\psi_{1}$ is a natural transformation.  

The composite of the arrows along the top and down the right side is, by definition, $(\psi_{2}\Gamma_{2})\varphi_{2}$.
The other path is $F$ applied to the string of inequalities 
\[
	\Iota \leq \Kappa_{1}\Gamma_{1} \leq \Kappa_{1}\Gamma_{2} \leq \Kappa_{2}\Gamma_{2},
\]
which is the unique natural transformation $\eta_{\Kappa_{2}\Gamma_{2}}$.
Since the diagram commutes, $(\psi_{2}\Gamma_{2})\varphi_{2} = F\eta_{\Kappa_{2}\Gamma_{2}}$.

A similar calculation gives $(\varphi_{2}\Kappa_{2})\psi_{2} = G \eta_{\Gamma_{2}\Kappa_{2}}$. Thus $(\varphi_{2},\psi_{2})$ is a $(\Gamma_{2},\Kappa_{2})$-interleaving.
\end{proof}

\begin{proposition}[triangle inequality]
\label{prop:triangle}
Let $\Gamma_1$, $\Gamma_2$, $\Kappa_1$, $\Kappa_2$ be translations and let $F, G, H$ be persistence modules over a preordered set~$\cat{P}$.
Then the statements
\begin{align*}
&\text{$F$ and $G$ are $(\Gamma_1,\Kappa_1)$-interleaved}
\\
&
\text{$G$ and $H$ are $(\Gamma_2,\Kappa_2)$-interleaved}
\end{align*}
together imply
\[
\text{$F$ and $H$ are $(\Gamma_2 \Gamma_1, \Kappa_1 \Kappa_2)$-interleaved}.
\]
\end{proposition}

\begin{proof}
Let $(\varphi_{1},\psi_{1})$ be a $(\Gamma_{1},\Kappa_{1})$-interleaving of $F$ and $G$, and let $(\varphi_{2},\psi_{2})$ be a $(\Gamma_{2},\Kappa_{2})$-interleaving of $G$ and $H$.  
\begin{equation*}
    \xymatrix{
    \cat{P} \ar[r]^{\Gamma_1} \ar[d]_F \ar@<1ex>@{}[dr]|*+{\stackrel{\mathlarger{\varphi_1}}{\To}} 
& \cat{P} \ar[d]^G \ar[r]^{\Gamma_2}
\ar@<1ex>@{}[dr]|*+{\stackrel{\mathlarger{\varphi_2}}{\To}}  
& \cat{P} \ar[d]^H \ar[r]^{\Kappa_2}
\ar@<1ex>@{}[dr]|*+{\stackrel{\mathlarger{\psi_2}}{\To}}  
& \cat{P} \ar[d]^G \ar[r]^{\Kappa_1}
\ar@<1ex>@{}[dr]|*+{\stackrel{\mathlarger{\psi_1}}{\To}}  
& \cat{P} \ar[d]^F \\ 
   \cat{D} \ar@{=}[r] 
& \cat{D} \ar@{=}[r] 
& \cat{D} \ar@{=}[r] 
& \cat{D} \ar@{=}[r] 
& \cat{D}
    }
\end{equation*}
 Define $\varphi_{3}$ and~$\psi_{3}$ as the composites
\[
	F \stackrel{\varphi_{1}}{\Longrightarrow} G\Gamma_{1} \stackrel{\varphi_{2}\Gamma_{1}}{\Longrightarrow} H \Gamma_{2} \Gamma_{1}
\]
and
\[
	H \stackrel{\psi_{2}}{\Longrightarrow} G \Kappa_{2} \stackrel{\psi_{1}\Kappa_{2}}{\Longrightarrow} F \Kappa_{1} \Kappa_{2}
\]
respectively.  

Consider the diagram
\[
	\xymatrix{
		F \ar@2[dr]_{\varphi_{1}} \ar@2[rr]^{F\eta_{\Kappa_{1}\Gamma_{1}}}
			& & F\Kappa_{1}\Gamma_{1}  \ar@2[rr]^{F\Kappa_{1} \eta_{\Kappa_{2}\Gamma_{2}}\Gamma_{1}}
			& & F\Kappa_{1} \Kappa_{2} \Gamma_{2} \Gamma_{1} \\
		& G \Gamma_{1} \ar@2[ur]^{\psi_{1}\Gamma_{1}} \ar@2[rr]^{G\eta_{\Kappa_{2}\Gamma_{2}}\Gamma_{1}} \ar@2[dr]_{\varphi_{2}\Gamma_{1}}
			& & G \Kappa_{2} \Gamma_{2} \Gamma_{1} \ar@2[ur]_{\psi_{1}\Kappa_{2} \Gamma_{2} \Gamma_{1}} \\ 
		& &  H \Gamma_{2} \Gamma_{1}			 \ar@2[ur]_{\psi_{2}\Gamma_{2}\Gamma_{1}} 			 
	}
\]
of functors and natural transformations.  The two triangles commute because $(\varphi_{i},\psi_{i})$ is a $(\Gamma_{i},\Kappa_{i})$-interleaving, for $i = 1,2$.  The parallelogram commutes because $\psi_{1}$ is a natural transformation, and so the full diagram commutes. 

Composing first down to the right, then up to the right, yields $(\psi_{3} \Gamma_{2}\Gamma_{1}) \varphi_{3}$.  Since
\[
	(\Kappa_{1} \eta_{\Kappa_{2}\Gamma_{2}} \Gamma_{1})\eta_{\Kappa_{1}\Gamma_{1}} = \eta_{\Kappa_{1} \Kappa_{2} \Gamma_{2} \Gamma_{1}}
\]
and $F$ respects composition, the composite along the top is $F \eta_{\Kappa_{1} \Kappa_{2} \Gamma_{2} \Gamma_{1}}$.  It follows that $(\psi_{3} \Gamma_{2}\Gamma_{1}) \varphi_{3} = F \eta_{\Kappa_{1} \Kappa_{2} \Gamma_{2} \Gamma_{1}}$.  

Similarly, $(\varphi_{3} \Kappa_{2} \Kappa_{1}) \psi_{3} = G  \eta_{\Gamma_{2} \Gamma_{1} \Kappa_{1} \Kappa_{2} }$. Thus $(\varphi_{3}, \psi_{3})$ is a $(\Gamma_{2} \Gamma_{1}, \Kappa_{1} \Kappa_{2})$-interleaving.
\end{proof}

These notions give us rather a lot of relationships between the elements of~$\cat{D}^{\cat{P}}$. As a measure of similarity between two persistence modules, one can consider either of:
\begin{align*}
\left\{ \Gamma \mid \text{$F,G$ are $\Gamma$-interleaved} \right\}
&\subseteq \Trans_{\cat{P}}
\\
\left\{ (\Gamma, \Kappa)
\mid
\text{$F,G$ are $(\Gamma,\Kappa)$-interleaved} \right\}
&\subseteq \Trans_{\cat{P}} \times \Trans_{\cat{P}}
\end{align*}
By Proposition~\ref{prop:monotonicity}, these are `up-sets' or `neighbourhoods of~$\infty$' in the respective posets that contain them. The larger the neighbourhood, the more similar are $F$ and~$G$. But this is not easy to work with. 
Here are two strategies for simplifying this morass of information:
\begin{vlist}
\item
(Section~\ref{sec:sublinear})
Project $\Trans_{\cat{P}}$ onto the monoid $[0, \infty]$.
\item
(Section~\ref{sec:superlinear})
Embed the monoid $[0, \infty)$ in $\Trans_{\cat{P}}$.
\end{vlist}
Either strategy allows us to define a metric on $\cat{D}^{\cat{P}}$. In some cases the resulting metrics are equal. In particular, either method can be used to recover the usual interleaving distance for modules over $\cat{R} = (\R,\leq)$.

\subsection{Sublinear projections}
\label{sec:sublinear}

Let $\cat{P} = (P, \leq)$ be a preordered set. A \emph{sublinear projection} is a function $\omega: \Trans_{\cat{P}} \to [0,\infty]$ such that
\begin{vlist}
\item
$\omega_\Iota = 0$, where $\Iota$ is the identity translation; and

\item
(sublinearity)
$\omega_{\Gamma_1 \Gamma_2}
\leq
\omega_{\Gamma_1} + \omega_{\Gamma_2}$
for all $\Gamma_1, \Gamma_2$.

\end{vlist} 

More stringently, a sublinear projection is \emph{monotone} if $\omega_\Gamma \leq \omega_\Kappa$ whenever $\Gamma \leq \Kappa$. We largely work without this condition, although most of the sublinear projections we encounter are indeed monotone. If needed, a sublinear projection that isn't monotone may be replaced by its `monotone hull'
\[
\overline{\omega}_\Gamma
=
\inf \left( \omega_{\Gamma'} \mid \Gamma' \geq \Gamma \right)
\]
which is a monotone sublinear projection.

Given a sublinear projection, it is a simple matter to define an interleaving distance on persistence modules over~$\cat{P}$ in an arbitrary category~$\cat{D}$.

\begin{definition}
We say that a translation $\Gamma$ is an \emph{$\eps$-translation} if $\omega_\Gamma \leq \eps$.
Persistence modules $F, G \in \cat{D}^{\cat{P}}$ are \emph{$\eps$-interleaved with respect to~$\omega$} if they are $(\Gamma, \Kappa)$-interleaved for some pair of $\eps$-translations $\Gamma, \Kappa$.
\end{definition}

%

\begin{definition}
The \emph{interleaving distance} between $F,G \in \cat{D}^{\cat{P}}$ is
\begin{equation}
\meto(F,G)
=
\inf
\left\{ \eps \in [0, \infty)
\mid
\text{$F,G$ are $\eps$-interleaved with respect to~$\omega$} \right\}.
\end{equation}
We define $\meto(F,G) = \infty$ if the set on the right hand side is empty.
Notice that the set whose infimum we take must be an interval of the form $(t,\infty)$ or $[t,\infty)$ if it is not empty. This is because the statement ``$F,G$ are $\eps$-interleaved'' becomes weaker as $\eps$ increases.
\end{definition}

We leave to the reader the following corollary Proposition~\ref{prop:monotonicity} (monotonicity):

\begin{proposition}
Let $\overline{\omega}$ be the monotone hull of~$\omega$. Then $\meto = \met^{\overline{\omega}}$.
\qed
\end{proposition}

This is where we have been heading:

\begin{theorem}
\label{thm:inter-sub}
The interleaving distance $\met = \meto$ is an extended pseudometric on $\cat{D}^{\cat{P}}$.
\end{theorem}

\begin{proof}
Clearly $\met$ is symmetric; and $\met(F,F) = 0$ because $F$ is $\Iota$-interleaved with itself.
It remains to verify the triangle inequality.
Suppose $F, G$ are $\eps_1$-interleaved and $G,H$ are $\eps_2$-interleaved. Then we can find $\eps_j$-translations $\Gamma_j, \Kappa_j$ (for $j = 1,2$) such that  $F,G$ are $(\Gamma_1, \Kappa_1)$-interleaved and $G,H$ are $(\Gamma_2, \Kappa_2)$-interleaved.
By Proposition~\ref{prop:triangle} it follows that $F, H$ are $(\Gamma_2 \Gamma_1, \Kappa_1 \Kappa_2)$-interleaved.
Sublinearity implies
\begin{alignat*}{2}
\omega_{\Gamma_2 \Gamma_1}
&\leq \omega_{\Gamma_2} + \omega_{\Gamma_1}
&& \leq \eps_2 + \eps_1
\\
\omega_{\Kappa_1 \Kappa_2}
&\leq \omega_{\Kappa_1} + \omega_{\Kappa_2}
&& \leq \eps_1 + \eps_2
\end{alignat*}
so $F,H$ are $(\eps_1+\eps_2)$-interleaved with respect to~$\omega$. The result now follows by letting $\eps_1, \eps_2$ approach their respective infima.
\end{proof}

\begin{theorem}
\label{thm:sub-stability}
Let $F,G \in \cat{D}^{\cat{P}}$ and let $H: \cat{D} \to \cat{E}$. Then $\met(HF,HG) \leq \met(F,G)$.
\end{theorem}

\begin{proof}
If $F,G$ are $(\Gamma, \Kappa)$-interleaved for some pair $\Gamma, \Kappa$ of $\eps$-translations, then $HF,HG$ are also $(\Gamma,\Kappa)$-interleaved, by functoriality.
The result follows from this implication.
\end{proof}

We thank an anonymous referee for suggesting the following example.

Let $\cat{V}$ denote the category of finite-dimensional vector spaces over a field~$\F$. The operation of taking the dual of a vector space defines a functor $[\cdot]^*: \cat{V} \to \cat{V}^{\mathrm{op}}$, where $\cat{V}^{\mathrm{op}}$ denotes the category which is `{opposite}' to~$\cat{V}$ in the sense of having all the morphisms reversed. It can equally well be interpreted as a functor $[\cdot]^*: \cat{V}^{\mathrm{op}} \to \cat{V}$.

\begin{example}
Let $\cat{P}$ be a preordered set with a sublinear projection, so that we have interleaving distances~$\met$ for every functor category on~$\cat{P}$. The dual of a module $F: \cat{P} \to \cat{V}$ is a module $F^*: \cat{P} \to \cat{V}^{\mathrm{op}}$. Moreover, given two such modules $F,G : \cat{P} \to \cat{V}$ we have
\[
\met(F,G) = \met(F^*,G^*).
\]
Thus dualization defines an isometry $\cat{V}^{\cat{P}} \to (\cat{V}^{\mathrm{op}})^{\cat{P}}$.
\end{example}

\begin{proof}
By Theorem~\ref{thm:sub-stability} we have
\[
\met(F,G) \geq \met(F^*,G^*) \geq \met(F^{**},G^{**}).
\]
The well-known natural isomorphism between a finite-dimensional vector space~$V$ and its double dual~$V^{**}$ gives rise to isomorphisms $F \cong F^{**}$ and $G \cong G^{**}$ so all three expressions must be equal.
\end{proof}

The result does \emph{not} work if we regard $F^*$ as a functor $\cat{P}^{\mathrm{op}} \to \cat{V}$ instead. In general, the semigroups of translations $\Trans_\cat{P}$ and $\Trans_{\cat{P}^\mathrm{op}}$ may be quite different, and there is no general way to transfer interleaving metrics from one domain to the other. Categories of modules over a preordered set and its opposite have no reason to behave similarly (except in special instances such as $\cat{P} = \cat{R}$). The direction of the preordered set matters.

\subsection{Lawvere metrics}
\label{sec:lawvere}
We now identify a rich source of sublinear projections.

A \emph{Lawvere metric space}\footnote{%
Lawvere spaces are also called extended quasipseudometric spaces.%
}%
~\cite{lawvere:metric} is a set $P$ together with a function $d: P \times P \to [0,\infty]$ such that $d(x,x)=0$ for all $x \in P$, and $d(x,z) \leq d(x,y) + d(y,z)$ for all $x,y,z \in P$.
Note that there is no symmetry condition and the distance from one object to a different object may be zero or infinite.

\begin{proposition} \label{prop:sublinear}
Let $\cat{P} = (P, \leq)$ be a preordered set, and let $d$ be a Lawvere metric on~$P$. Then the formula
\[
\omega_\Gamma = 
\sup \left( d(x, \Gamma(x)) \mid x \in P \right)
\]
defines a sublinear projection~$\omega = \omega^d$.
\end{proposition}

\begin{proof}
Clearly $\omega_\Gamma \in [0,\infty]$ and $\omega_I = 0$. For sublinearity, note that
\[
d(x, \Gamma_1 \Gamma_2 (x))
\leq
d(x, \Gamma_2(x)) + d(\Gamma_2(x), \Gamma_1\Gamma_2(x))
\leq
\omega_{\Gamma_2} + \omega_{\Gamma_1}
\]
for all~$x$, and therefore $\omega_{\Gamma_1 \Gamma_2} \leq \omega_{\Gamma_2} + \omega_{\Gamma_1}$ as required.
\end{proof}

Therefore, using Theorem~\ref{thm:inter-sub}, any choice of Lawvere metric on~$\cat{P}$ defines an interleaving distance on $\cat{D}^{\cat{P}}$. 
It may be helpful to unpack the definitions here:

\begin{vlist}
\item
$F,G$ satisfy $\met(F,G) \leq \eps$ if and only if they are $\eta$-interleaved for all $\eta > \eps$.

\item
$F,G$ are $\eta$-interleaved if and only if they are $(\Gamma, \Kappa)$-interleaved for some $\Gamma, \Kappa$ which satisfy the conditions
\[
d(x, \Gamma(x)) \leq \eta,
\quad
d(x, \Kappa(x)) \leq \eta
\]
for all $x \in P$.
\end{vlist}

\begin{example}
[continuing Examples \ref{ex:real-Omega} and~\ref{ex:real-Omega2}]
\label{ex:real-Omega3}
Let $\omega$ be the sublinear projection associated with the standard metric $d(s,t) = |s-t|$ on the real line.
If $\Gamma$ is a translation with $\omega_\Gamma \leq \eps$, then
\[
t \leq \Gamma(t) \leq t + \eps
\]
for all $t \in \R$, and so $\Gamma \leq \Omega_\eps$. Therefore, using Proposition~\ref{prop:monotonicity}:
\begin{align*}
&
\text{$F,G$ are $\eps$-interleaved with respect to~$\omega$}
\\
&\Leftrightarrow\;
\text{$F,G$ are $(\Gamma,\Kappa)$-interleaved for some $\Gamma,\Kappa$ with $\omega_{\Gamma},\omega_{\Kappa} \leq \eps$}
\\
&\Leftrightarrow\;
\text{$F,G$ are $\Omega_\eps$-interleaved}
\end{align*}
We conclude that $\met = \met^\omega$ is the usual interleaving distance for persistence modules over~$\cat{R}$.
\end{example}

\subsection{Superlinear families}
\label{sec:superlinear}

In Sections \ref{sec:sublinear} and~\ref{sec:lawvere}, we constructed interleaving distances on~$\cat{D}^{\cat{P}}$ by projecting $\Trans_{\cat{P}}$ sublinearly onto $[0,\infty]$. Now we consider the dual approach.

Let $\cat{P} = (P, \leq)$ be a preordered set. A \emph{superlinear family} is a 
function $\Omega: [0, \infty) \to \Trans_{\cat{P}}$, which
satisfies
$\Omega_{\eps_1 + \eps_2} \geq \Omega_{\eps_1} \Omega_{\eps_2}$
for all $\eps_1, \eps_2 \geq 0$.
%

Here are two consequences of the definition:
\begin{vlist}
\item
The family $(\Omega_\eps)$ is monotone. Indeed,
\[
\Omega_{\eps_1}
=
\Iota
\Omega_{\eps_1}
\leq
\Omega_{\eps_2 - \eps_1}
\Omega_{\eps_1}
\leq
\Omega_{\eps_2}
\]
whenever $\eps_1 \leq \eps_2$.

\item
If $\cat{P}$ is a poset\footnote{%
For preordered sets, the statement is that there exist natural isomorphisms between the functors $\Omega_0^2$ and~$\Omega_0$, and between the functors $\Omega_0 \Omega_{\eps}$, $\Omega_{\eps} \Omega_0$ and~$\Omega_{\eps}$.
},
then $\Omega_0$ is idempotent and fixes $(\Omega_\eps)$ on both sides. Indeed,
\[
\Omega_\eps
= \Iota \Omega_\eps
\leq \Omega_0 \Omega_\eps
\leq \Omega_\eps
\quad
\text{and}
\quad
\Omega_\eps
= \Omega_\eps \Iota
\leq \Omega_\eps \Omega_0
\leq \Omega_\eps
\]
so $\Omega_0 \Omega_\eps = \Omega_\eps \Omega_0 = \Omega_\eps$ for all $\eps \geq 0$, and in particular $\Omega_0^2 = \Omega_0$.
\end{vlist}

These are useful clues when seeking such a family.

Let $\cat{P}$ be a preordered set with a superlinear family $\Omega$, and let $\cat{D}$ be an arbitrary category.
\begin{definition}
The \emph{interleaving distance} between $F,G \in \cat{D}^{\cat{P}}$ is
\begin{equation}
\metO(F,G)
=
\inf
\left\{ \eps \in [0, \infty)
\mid
\text{$F,G$ are $\Omega_\eps$-interleaved} \right\}.
\end{equation}
We define $\met^\Omega(F,G) = \infty$ if the set on the right-hand side is empty.
\end{definition}

\begin{theorem}
\label{thm:inter-super}
The interleaving distance $\met = \metO$ is an extended pseudometric on $\cat{D}^{\cat{P}}$.
\end{theorem}

\begin{proof}
Clearly $\met$ is symmetric; and $\met(F,F) = 0$ because $F$ is $\Iota$-interleaved, and therefore by Proposition~\ref{prop:monotonicity} $\Omega_0$-interleaved, with itself.
It remains to verify the triangle inequality.
Suppose that $F,G$ are $\Omega_{\eps_1}$-interleaved, and $G,H$ are $\Omega_{\eps_2}$-interleaved.
By Proposition~\ref{prop:triangle} it follows that $F, H$ are $(\Omega_{\eps_2} \Omega_{\eps_1}, \Omega_{\eps_1} \Omega_{\eps_2})$-interleaved.
By Proposition~\ref{prop:monotonicity} and the superlinearity condition, it follows that $F, H$ are $\Omega_{\eps_1+\eps_2}$-interleaved. 
The result follows by letting $\eps_1$ and $\eps_2$ approach their respective infima.
\end{proof}

\begin{example}
\label{ex:real-2}
The maps $\Omega_\eps: t \mapsto t + \eps$ comprise a superlinear family in $\cat{R} = (\R, \leq)$.
Its interleaving distance $\metO$ is the standard one.
\end{example}

\begin{theorem}
\label{thm:super-stability}
Let $F,G \in \cat{D}^{\cat{P}}$ and let $H: \cat{D} \to \cat{E}$. Then $\met(HF,HG) \leq \met(F,G)$.
\end{theorem}

\begin{proof}
If $F,G$ are $\Omega_\eps$-interleaved then $HF,HG$ are $\Omega_\eps$-interleaved, by functoriality.
\end{proof}

A comparison of Examples \ref{ex:real-Omega3} and~\ref{ex:real-2} suggests how to get the best of both worlds.

\begin{theorem}
\label{thm:equality}
Let $\omega$ be a sublinear projection on a preordered set~$\cat{P}$. Suppose for every $\eps \geq 0$ there exists a translation $\Omega_\eps$ with $\omega_{\Omega_\eps} \leq \eps$,
which is `largest' in the sense that $\omega_\Gamma \leq \eps$ implies $\Gamma \leq \Omega_\eps$.
Then $(\Omega_\eps)$ is a superlinear family, and the two interleaving distances are equal: $\meto = \metO$.
\end{theorem}

\begin{proof}
For superlinearity, we have
\[
\omega_{\Omega_{\eps_1}\Omega_{\eps_2}}
\leq
\omega_{\Omega_{\eps_1}}
+
\omega_{\Omega_{\eps_2}}
\leq
\eps_1 + \eps_2
\]
which implies that $\Omega_{\eps_1}\Omega_{\eps_2} \leq \Omega_{\eps_1 + \eps_2}$ 
by the `largest' property.
To compare distances, notice that the following statements are equivalent:
\begin{align*}
&\text{$F,G$ are $(\Gamma,\Kappa)$-interleaved for some $\Gamma, \Kappa$ such that $\omega_\Gamma, \omega_\Kappa \leq \eps$}
\\
\;\Leftrightarrow\;\;
&\text{$F,G$ are $(\Omega_\eps, \Omega_\eps)$-interleaved}
\end{align*}
The implication `$\Leftarrow$' is trivial, and the implication `$\Rightarrow$' follows by Proposition~\ref{prop:monotonicity} and the fact that $\Omega_\eps$ is `largest'. The equality $\meto = \metO$ follows upon taking the infimum over the set of~$\eps$ for which each statement is true.
\end{proof}

The relationship between $\omega, \Omega$ can be understood at a more abstract level. One clue is that the inequality $\omega_{\Omega_{\eps}} \leq  \eps$ in the statement of Theorem~\ref{thm:equality} takes the form of a co-unit of an adjunction. This suggests that we can construct either of $\omega, \Omega$ from the other using a version of the Adjoint Functor Theorem in category theory.
We explore this in Section~\ref{subsec:monoids}.

We finish Section~\ref{sec:interleaving} with an observation. If $\cat{P}$ is a preordered set equipped with $\Omega$ or~$\omega$, then our results combine to give a functor ${[\cdot]}^\cat{P}: \cat{Cat} \to \cat{Met}$, where $\cat{Cat}$ is the category of small categories and functors, and $\cat{Met}$ is the category of extended pseudometric spaces and 1-Lipschitz maps.

\section{Functions into a metric space}
\label{sec:map-to-metric}

\subsection{Subsets of a metric space}
\label{sec:subsets}

Let $(M, d_M)$ be a metric space, or more generally a Lawvere metric space, and let $\Subs_M$ denote the poset of subsets of~$M$. By considering generalized persistence modules over subposets $\cat{P} \subseteq \Subs_M$, we recover several standard theories of persistence.

A natural Lawvere metric on $\Subs_M$ is the asymmetric Hausdorff distance
\[
\antihaus(A,B)
=
\sup_{b \in B} \left[ \inf_{a \in A} \left( d_M(a,b) \right)\right]
\]
which measures how far $B$ escapes from~$A$. 
The two Lawvere conditions $\antihaus(A,A)=0$ and $\antihaus(A,C) \leq \antihaus(A,B) + \antihaus(B,C)$ follow from the corresponding conditions on~$d_M$.

More familiar, perhaps, is the (symmetric) Hausdorff distance
\[
\haus(A,B) = \max(\ell_{\rm H}(A,B), \ell_{\rm H}(B,A))
\]
on $\Subs_M$.

Let $\cat{P} \subseteq \Subs_M$ be a subposet. The method of Section~\ref{sec:lawvere} provides sublinear projections
\begin{alignat*}{2}
\omega^\ell_\Gamma
&= \sup \left( \antihaus(A, \Gamma(A)) \mid A \in \cat{P} \right)
\\
\omega^d_\Gamma
&= \sup \left( \haus(A, \Gamma(A)) \mid A \in \cat{P} \right)
\end{alignat*}
on $\cat{Trans_P}$, and hence interleaving distances on any $\cat{D}^\cat{P}$.

The next proposition saves us from having to choose.

\begin{proposition}
The sublinear projections $\omega^\ell, \omega^d$ are equal.
\end{proposition}

\begin{proof}
For any $A \subseteq B$ we have $\antihaus(B,A) = 0$ and therefore $\haus(A,B) = \antihaus(A,B)$.
Since $A \subseteq \Gamma(A)$ it follows that $\haus(A,\Gamma(A)) = \antihaus(A, \Gamma(A))$ for all~$A$.
\end{proof}

The interleaving distance defined by the sublinear projection $\omega = \omega^\ell = \omega^d$ is called the \emph{induced metric} on $\cat{D}^\cat{P}$; specifically, it is induced by~$d_M$. It is an extended pseudometric.

There is a largest set $B$ such that $\antihaus(A,B) \leq \eps$. This is the \emph{$\eps$-offset of~$A$}, defined
\[
A^\eps
=
\left\{ m \in M \mid {\textstyle \inf_{a \in A}} (d_M(a,m)) \leq \eps \right\}
=
\left\{ m \in M \mid \antihaus(A, \{m\}) \leq \eps \right\}.
\]
The operation $\Omega_\eps: A \mapsto A^\eps$ is a translation on $\Subs_M$, and Theorem~\ref{thm:equality} implies that $(\Omega_\eps)$ is a superlinear family and $\met^\Omega = \met^\omega$.

\subsection{Inverse-image diagrams}
\label{sec:inverse-image}
Let $X,M$ be sets. Then any function $f: X \to M$ induces a map of posets
$\pre{f}: \Subs_M \to \Subs_X$
\[
\pre{f}(A) = \left\{ x \in X \mid f(x) \in A \right\},
\]
the \emph{inverse-image} map. This is order-preserving, so it is a functor.

Fix a collection of subsets $\cat{P} \subseteq \Subs_M$, and let $X$ be a topological space. Then $f$ induces a functor
\[
\raisebox{-0.2ex}{$F$ :}
\begin{diagram}
\node{\cat{P}}
	\arrow{e,t}{\subseteq}
\node{\Subs_M}
	\arrow{e,t}{\pre{f}}
\node{\Subs_X}
	\arrow{e,b}{}
\node{\cat{Top}}
\end{diagram}
\]
where the last arrow is the functor that `remembers' the subspace topology inherited by each subset of~$X$. (It is a functor because the inclusion of one subspace in another is continuous.)

We can think of this process as a function $\Inv: M^X \to \cat{Top}^{\cat{P}}$ which takes the function~$f$ to its \emph{inverse-image diagram}~$F$, a generalized persistence module in~$\cat{Top}$ over~$\cat{P}$.

Let $d_M$ be a Lawvere metric on~$M$. Then:

\begin{vlist}
\item
$d_M$ induces an extended pseudometric $\met_\infty$ on $M^X$, by the formula
\begin{align*}
\met_\infty(f,g)
&= \max ( \hat\met(f,g), \hat\met(g,f) )
\\
\text{where}\quad
\hat{\met}_\infty(f,g)
&= \sup \left( d_M(f(x), g(x)) \mid x \in X \right)
\end{align*}
for functions $f,g: X \to M$.

\item
$d_M$ induces an extended pseudometric $\met = \met^\omega$ on $\cat{Top}^{\cat{P}}$ (Section~\ref{sec:subsets}). 
\end{vlist}

How does $\Inv$ behave with respect to these two metrics?

\begin{definition}
\label{def:enough}
We say that $\cat{P} \subseteq \Subs_M$ has \emph{enough translations} if for every $0 \leq \eps < \eta$ there exists $\Gamma = \Gamma_{\eps, \eta} \in \cat{Trans_P}$ with $\omega_{\Gamma} \leq \eta$, such that for every $A \in \cat{P}$ we have an inclusion
\[
\hat{A}^\eps
=
\{ m \in M \mid \text{there exists $a \in A$ such that $d_M(a,m) \leq \eps$} \}
\,\subseteq\,
\Gamma(A).
\]
We call $\hat{A}^\eps$ the \emph{weak $\eps$-offset} of~$A$. It is contained in the offset~$A^\eps$ and not always equal to it.
\end{definition}

If $\cat{P}$ is closed under taking $\eps$-offsets then it has enough translations; let $\Gamma_{\eps,\eta} = (A \mapsto A^\eps)$.
Examples include the following:
\begin{vlist}
\item
$\cat{P} = \Subs_M$.

\item
$\cat{P} = \{ \text{closed subsets of~$M$} \}$ when $M$ is a metric space.

\item
$\cat{P} = \{ \text{compact subsets of~$M$} \}$ when $M$ is a proper\footnote{%
A metric space is `proper' if its closed bounded subsets are compact.
}
metric space.
\end{vlist}
Similarly, if $\cat{P}$ is closed under taking weak $\eps$-offsets then it has enough translations; let $\Gamma_{\eps,\eta} = (A \mapsto \hat{A}^\eps)$.

Here is an example where the $\eta$~parameter comes into play.
\begin{vlist}
\item
$\cat{P} = \{ \text{open subsets of~$M$} \}$ when $M$ is a metric space.
\end{vlist}
Let $\Gamma_{\eps,\eta} = (A \mapsto \intr(A^\eta))$.
Observe that $A^\eta$ contains an open neighbourhood of $\hat{A}^\eps$, namely the union of open $(\eta-\eps)$-balls centered at the points of~$\hat{A}^\eps$. It follows that $\intr(A^\eta)$ contains this open set and therefore $\hat{A}^\eps$, as required.

The reader may verify that there is no guaranteed inclusion $\hat{A}^\eps \subseteq \intr (A^\eps)$.

\begin{theorem}[Inverse-image stability]
\label{thm:inv-stab}
Let $X$ be a topological space, and let $(M, d_M)$ be a Lawvere metric space.
Suppose that $\cat{P} \subseteq \Subs_M$ has enough translations.
Then for any $f, g: X \to M$ and $F = \Inv(f)$, $G = \Inv(g)$, we have
\[
\met(F,G) \leq \met_\infty(f,g).
\]
That is, $\Inv: M^X \to \cat{Top}^\cat{P}$ is $1$-Lipschitz (non-expanding).
\end{theorem}

\begin{proof}
Write $\eps = \met_\infty(f,g)$. Let $\eta > \eps$ and take a corresponding~$\Gamma = \Gamma_{\eps,\eta}$.
Then
\begin{alignat*}{2}
f(x) \in A
&\; \To \;
g(x) \in \hat{A}^\eps
&&\; \To \;
g(x) \in \Gamma(A)
\\
g(x) \in A
&\; \To \;
f(x) \in \hat{A}^\eps
&&\; \To \;
f(x) \in \Gamma(A)
\end{alignat*}
where the first~$\To$ follow from $\met_\infty(f,g) = \eps$ and the second~$\To$ follow from the choice of~$\Gamma$.
The resulting inclusions
\[
\pre{f}(A) \subseteq \pre{g} \Gamma(A),
\quad
\pre{g}(A) \subseteq \pre{f} \Gamma(A)
\]
imply that the functors
\[
\begin{diagram}
\node{\cat{P}}
	\arrow{e,t}{\subseteq}
\node{\Subs_M}
	\arrow{e,t}{\pre{f}}
\node{\Subs_X}
\\
\node{\cat{P}}
	\arrow{e,t}{\subseteq}
\node{\Subs_M}
	\arrow{e,t}{\pre{g}}
\node{\Subs_X}
\end{diagram}
\]
are $\Gamma$-interleaved, by Lemma~\ref{lem:poset}.
Therefore $F, G$ are $\Gamma$-interleaved, by Proposition~\ref{prop:functorial}.
Thus $\met(F,G) \leq \omega_\Gamma \leq \eta$, and the result follows by letting $\eta \to \eps$.
\end{proof}

\begin{remark*}
Suppose $\Gamma = \Gamma_{\eps,\eta}$ can be chosen independently of $\eta$. This implies $\omega_\Gamma \leq \eps$, and we get the sharper conclusion that $F,G$ are $\eps$-interleaved, not merely $\eta$-interleaved for all $\eta > \eps$.
For instance, this happens when $\cat{P}$ is closed under $\eps$-offsets.\end{remark*}

Combining Theorems \ref{thm:inv-stab} and~\ref{thm:sub-stability}:

\begin{theorem}[inverse-image stability]
\label{thm:H-inv-stab}
Let $H: \cat{Top} \to \cat{E}$ be any functor. Then
\[
\met(HF,HG) \leq \met_\infty(f,g)
\]
in the situation of Theorem~\ref{thm:inv-stab}.
\qed
\end{theorem}

\subsection{Inverse-image persistence theories}
\label{sec:inverse-image-theories}

We give several examples of persistence theories which fit into the framework developed in sections \ref{sec:subsets}--\ref{sec:inverse-image}. For the familiar forms of these theories, suppose that $H$ is a homology functor $\cat{Top} \to \Vect_\F$.

\smallskip
\begin{example}[sublevelset persistence~\cite{CohenSteiner_E_H_2007,Chazal_CS_G_G_O_2009,bubenikScott:1}]
\label{ex:sublevelset}
Let the target space be $\R$ with the standard metric, and let $\cat{P}$ be the poset of intervals of the form $I^t = (-\infty, t]$ where $t \in \R$.
If $f : X \to\R$ then $F = \Inv(f)$ is the sublevelset filtration of~$X$, and $HF$ is the sublevelset persistent homology.
Since
\[
[I^t]^\eps = I^{t+\eps}
\]
it follows that $\cat{P}$ has enough translations and $\norm{f - g}_\infty \leq \eps$ implies that $HF, HG$ are $\eps$-interleaved.

We normally identify the poset $\cat{P}$ with the real line~$\R$ in the obvious way. If we extend $\cat{P}$ to include the interval $(-\infty, +\infty)$, then the identification is with $\R_+ = \R \cup \{ +\infty\}$.
\end{example}

\smallskip
\begin{example}[multidimensional sublevelset persistence~\cite{carlssonZomorodian:multidimP,Frosini:2013}] \label{ex:multi}
\label{ex:multi-d}
Let the target space be $\R^n$ with the $\ell_\infty$ metric, and let $\cat{P}$ be the poset of lower quadrants
\[
Q^{\vect{a}}
=
(-\infty, a_1] \times \cdots \times (-\infty, a_n]
\]
where $\vect{a} = (a_1, \dots, a_n) \in \R^n$.
If $\vect{f}: X \to \R^n$ then $F = \Inv(\vect{f})$ is a multi-filtration of~$X$, and $HF$ is the multi-sublevelset persistent homology.
Note that
\[
[Q^{\vect{a}}]^\eps
=
Q^{\vect{a}+\eps}
\]
where $\vect{a} + \eps = (a_1+\eps, \dots, a_n + \eps)$, so $\cat{P}$ has enough translations and $\norm{\vect{f} - \vect{g}}_\infty \leq \eps$ implies that $HF, HG$ are $\eps$-interleaved. Here $\norm{ \cdot }_\infty$ denotes the supremum over the domain of the function of the $\ell^\infty$-norm in~$\R^n$.

As a poset, $\cat{P}$ is isomorphic to $\R^n$ with the standard partial order.
If we allow $a_i = +\infty$, then the identification is with
$\R_+^n = (\R_+)^n$.
See Section~\ref{subsec:multi-d} for more ways to handle multidimensional persistence modules.
\end{example}

\smallskip
\begin{example}[levelset
  persistence~\cite{Carlsson_dS_M_2009,Bendich:2013}]
\label{ex:levelset}
Let the target space be $\R$ with the standard metric, and let $\cat{P}$ be the closed bounded intervals.
If $f:X \to \R$ then $F=\Inv(f)$ is the levelset diagram of~$f$, and $HF$ is the levelset persistent homology.
Note that
\[
[a,b]^\eps = [a-\eps, b+\eps]
\]
so $\cat{P}$ has enough translations and $\norm{ f - g }_\infty \leq \eps$ implies that $HF,HG$ are $\eps$-interleaved.

Here $\cat{P}$ may be identified with the planar region $\{ y \geq x \}$ and its nonstandard partial order
\[
(x_1,y_1) \leq (x_2, y_2)
\;\;
\Leftrightarrow
\;\;
\text{$x_1 \geq x_2$ and $y_1 \leq y_2$}.
\]
We can extend~$\cat{P}$ to include infinite intervals by allowing $x = -\infty$ and $y = +\infty$.
\end{example}

\smallskip 
\begin{example}[circular
  persistence~\cite{Dey_Burghelea_2011,Burghelea-Haller:2013}]
Let the target space be $\mathbb{S}^1 = \R/2\pi\Z$ with the metric inherited from $\R$. Let $\cat{P}$ be the poset of closed path-connected subsets of~$\mathbb{S}^1$.
If $f:X \to \mathbb{S}^1$ then $F=\Inv(f)$ is the circular levelset diagram of~$f$, and $HF$ is the angle-valued levelset persistent homology.
As with the previous examples, $\cat{P}$ is closed under taking $\eps$-offsets and $\met_\infty(f,g) \leq \eps$ implies that $HF, HG$ are $\eps$-interleaved.
\end{example}

\smallskip
\begin{example}[copresheaves on a metric space~\cite{curry:cosheaves}]
\label{ex:copresheaves}
Let $M$ be a metric space and let $\cat{P}$ be the poset of open subsets of~$M$. If $f:X \to M$ and $F=\Inv(f)$ then $HF$ is a copresheaf on~$M$.
Since $\cat{P}$ has enough translations, $\met_\infty(f,g) \leq \eps$ implies $\met(HF,HG) \leq \eps$.
\end{example}

We remind the reader that the bounds $\met(HF,HG)$ in the examples above are `soft' stability theorems, bounding the interleaving distance. 

The corresponding `hard' stability theorem for sublevelset persistence is well known. Let $X$ be a finite polyhedron and $f$ be continuous. Then the sublevelset persistent homology $HF$ is sufficiently tame that its persistence diagram may be defined. The interleaving distance between $HF,HG$ is then an upper bound for the bottleneck distance between the corresponding persistence diagrams~\cite{CohenSteiner_E_H_2007,Chazal_CS_G_G_O_2009,Chazal_dS_G_O_2012}.

In general, the question of a hard stability theorem arises for any choice of invariant descriptor for~$HF$ and a metric that compares different possible values of the descriptor.

\section{Monoidal structures}
\label{sec:monoids}

In Sections \ref{sec:sublinear} and~\ref{sec:superlinear} we showed that we could define an interleaving distance by comparing the monoid $\Trans_\cat{P}$ with the monoid $[0,\infty]$ by a sublinear projection $\omega:\Trans_{\cat{P}} \to [0,\infty]$ or with the monoid $[0,\infty)$ by a superlinear family $\Omega:[0,\infty) \to \Trans_{\cat{P}}$. In Section~\ref{subsec:monoids} we observe that $\omega$ and $\Omega$ are dual in a precise categorical sense. From this observation many of their properties follow easily. 

With this level of abstraction in hand, it is clear what is needed to
replace the monoids $[0,\infty]$ and $[0,\infty)$ in order to obtain other ways of
measuring interleavings. In Section~\ref{subsec:multi-d} we compare $\Trans_{\cat{P}}$ with
$[0,\infty]^n$ and $[0,\infty)^n$ and see that the resulting `vector
persistence' is stable.

\subsection{Monoidal adjunctions}
\label{subsec:monoids}

It is perhaps easier to work with a superlinear family than with a sublinear projection, since the latter entails considering the set of all possible translations rather than one specific family of translations.
On the other hand, we have seen that sublinear projections occur in a very natural way.
What is the relationship between the two methods?
In this subsection we give a brief high-level description of what is going on.

It turns out that the language of monoidal categories~\cite{maclane:book,kelly:doctrinal-adjunction} provides a pleasing interpretation of sublinear projections and superlinear families and their relationship to each other.

First we observe that $\cat{Trans_{\cat{P}}}$ is a strict monoidal category. Indeed it is a category since it is a preordered set. The tensor product is given by composition. We have $\Gamma \tensor \Kappa = \Gamma \Kappa$ and $(\Gamma \leq \Kappa) \tensor (\Lambda \leq \Mu) = \Gamma \Lambda \leq \Kappa \Mu$. The tensor unit is the identity $\Id \in \cat{\Trans_{\cat{P}}}$.
Since $\cat{Trans_{\cat{P}}}$ is a thin category all of the relations in the definition hold automatically.

Similarly, both $[0,\infty)$ and $[0,\infty]$ are strict monoidal categories with tensor product given by addition, $a \tensor b = a+b$ and $(a \leq b) \tensor (c \leq d) = a+b \leq c+d$, and tensor unit the number~$0$.

Now one can check that a superlinear family $\Omega: [0,\infty) \to \cat{Trans_{\cat{P}}}$ corresponds exactly to a lax monoidal functor $\Omega: ([0,\infty),+,0) \to (\cat{Trans_{\cat{P}}},\circ,\Id)$.
Indeed, monotone corresponds to functorial, and superlinearity
corresponds to the monoidal coherence maps. Note that $\Id \leq \Omega_0$ by definition.

In the other direction, a monotone sublinear projection $\omega: \cat{Trans_P} \to [0,\infty]$ corresponds exactly to an oplax monoidal functor $\omega: (\cat{Trans_P}, \circ, \Id) \to ([0,\infty],+,0)$.
Indeed, monotone corresponds to functorial, and the identity and
sublinearity conditions correspond to the coherence maps.

(There is no loss of generality in supposing that our sublinear projections are monotone, since we can always replace a sublinear projection by its monotone hull without changing the interleaving distance.)

\begin{definition}[adjunction relation]
Let $\omega: \cat{Trans_P} \to [0,\infty]$ and $\Omega: [0,\infty) \to \cat{Trans_P}$ be arbitrary functions.
We write $\omega \dashv \Omega$ if
\begin{equation}
\label{eq:adjoint-rel}
	\omega_{\Gamma} \leq \eps
	\;\Leftrightarrow\;
	\Gamma \leq \Omega_{\eps}
\end{equation}
holds for all $\eps \in [0,\infty)$ and $\Gamma \in \cat{Trans_{P}}$.
\end{definition}

For example, in the situation of Theorem~\ref{thm:equality}, if $\omega$ is monotone then $\omega \dashv \Omega$.

What we have is very nearly the notion of an adjoint pair of functors between the categories $\cat{Trans_P}$ and $[0,\infty]$.
In order theory, this is called a monotone Galois connection~\cite{continuousLatticesAndDomains}.
We say `very nearly' because the domain of~$\Omega$ does not
include~$\infty$. This can be fixed by appending a terminal object to
$\cat{Trans_P}$ if necessary and defining $\Omega_\infty$ to be this
terminal object.

As we see now, the adjunction relation guarantees that $\omega, \Omega$ are functors (i.e.\ are monotone), and that $\omega$ is lax monoidal if and only if $\Omega$ is oplax monoidal.

\begin{theorem}
\label{thm:adjoint-equality}
Let $\omega: \cat{Trans_P} \to [0,\infty]$ and $\Omega: [0,\infty) \to \cat{Trans_P}$ be arbitrary functions such that $\omega \dashv \Omega$.
Then:
\begin{vlist}
\item[(1)]
The `unit' inequality $\Gamma \leq \Omega_{\omega_\Gamma}$ holds for all $\Gamma \in \cat{Trans_P}$.

\item[(2)]
The `co-unit' inequality $\omega_{\Omega_\eps} \leq \eps$ holds for all $\eps \in [0, \infty)$.

\item[(3)]
The map $\omega$ is monotone.

\item[(4)]
The map $\Omega$ is monotone.
\end{vlist}
The next two statements are equivalent to each other.
\begin{vlist}
\item[(5)]
The map $\omega$ is a sublinear projection.
\item[(6)]
The map $\Omega$ is a superlinear family.
\end{vlist}
Suppose (5), (6) are true, so that the two interleaving metrics are defined.
\begin{vlist}
\item[(7)]
The metrics are equal: $\met^\omega = \met^\Omega$.
\end{vlist}
\end{theorem}

\begin{proof}
We invite the reader to find the very short proofs. Compare Theorem~\ref{thm:equality}.
\end{proof}

\begin{remark*}
Theorem~\ref{thm:adjoint-equality} and the preceding discussion explain, retroactively, why sublinear projections are a natural idea. Classical persistence uses interleaving distances based on the {linear} family $(\Omega_\eps: t \mapsto t + \eps)$ of translations on~$\R$. For subsets of a metric space, the analogous family of offset operations $(\Omega_\eps: A \mapsto A^\eps)$ is \emph{super}linear rather than linear: $(A^{\eps_1})^{\eps_2}$ is always contained in $A^{\eps_1 + \eps_2}$ but need not be equal to it. Thus we consider superlinear families.
Sublinear projections emerge as the dual concept to superlinear families.
\end{remark*}

The language of monoidal categories and adjunctions gives access to standard machinery for constructing $\omega$ from~$\Omega$ and vice versa.
Most immediately, the Adjoint Functor Theorem for preorders and doctrinal adjunction~\cite{kelly:doctrinal-adjunction} tells us that 
if $\Omega$ is a superlinear family which preserves infima, then 
\[
\omega_{\Gamma} = \inf ( \eps \st \Gamma \leq \Omega_{\eps})
\]
defines a sublinear projection with $\omega \dashv \Omega$.
Dually,
if $\cat{Trans_P}$ is a complete poset (meaning that every subset has an infimum and, equivalently, that every subset has a supremum), and $\omega$ is a sublinear projection that preserves suprema, then
\[
  \Omega_{\eps} = \sup \left( \Gamma \mid \omega_{\Gamma} \leq \eps\right)
\]
defines a superlinear family with $\omega \dashv \Omega$.
Completeness of $\cat{Trans_P}$ is used in a mild way, to guarantee the existence of the supremum.
If these particular suprema can be shown to exist in some other way, then the arguments of Theorem~\ref{thm:equality} lead to the same result.

\subsection{Vector persistence}
\label{subsec:multi-d}

The theory of multidimensional persistence is of great practical importance but as yet there is no consensus on how it is best handled. This is for essential reasons: the module theory is wild~\cite{carlssonZomorodian:multidimP} so no choice is completely satisfactory. In this section we apply our framework to multidimensional persistence modules. There is a larger monoid of translations than with classical persistence. This leads to many natural metrics $\met_\vect{a}$ which can be retrieved from a more complicated metric-like gadget~$\D$.
We give a brief sketch of these ideas, our goal being to indicate the range of possibilities rather than give a definitive recommendation.
For much more specific detailed investigations see~\cite{Biasotti_C_F_G_L_2009,Lesnick_2011}.

Recall that our metric theory for generalized persistence modules works by comparing the monoid of translations $\cat{\Trans_P}$ with the monoids $[0,\infty)$ and $[0,\infty]$ in terms of maps of the form:
\[
\begin{diagram}
\node{[0,\infty)}
	\arrow{e,tb}{\Omega}{\mathrm{lax}}
\node{\cat{\Trans_P}}
	\arrow{e,tb}{\omega}{\mathrm{oplax}}
\node{[0, \infty]}
\end{diagram}
\]
As we have seen, a map of either type will suffice, and a pair of maps satisfying the adjunction relation will give rise to the same metric.

Let us apply similar thinking to modules over the preordered set $\cat{R}^n = (\R^n, \leq)$, arising in multidimensional persistence~\cite{carlssonZomorodian:multidimP,Lesnick_2011}.
Given two persistence modules $F, G: \cat{R}^n \to \cat{D}$, how are we to compare them?

If numerical distances are sought, here are two natural approaches. Select a nonnegative vector $\vect{a} = (a_1, \dots, a_n)$ and an arbitrary vector $\vect{b} = (b_1, \dots, b_n)$.

\begin{vlist}
\item
Restrict $F, G$ to a line. Formally, let $L: \cat{R} \to \cat{R}^n$ be the functor defined by $t \mapsto \vect{b} + t \vect{a}$ and, using the standard metric for persistence modules over~$\cat{R}$, define
\[
\partial_{\vect{a}, \vect{b}}(F,G) = \met(FL, GL).
\]
This approach discards information carried by $F,G$ away from the chosen line.

\item
Alternatively, define a superlinear (in fact, linear) family by $\Omega_\eps(\vect{x}) = \vect{x} + \eps \vect{a}$ and set
\[
\met_\vect{a}(F,G) = \metO(F,G).
\]
This uses the structure of $F, G$ over the entire space. The comparison is made using translations in a single direction specified by~$\vect{a}$.

\end{vlist}

\begin{example}[restating Example~\ref{ex:multi-d}]
Instead of saying that $HF, HG$ are $\eps$-interleaved we can say that $\met_\vect{1}(HF,HG) \leq \eps$, where $\vect{1} = (1, \dots, 1)$.
\end{example}

\begin{remark*}
The second approach above is related to a metric defined on multidimensional size functions in~\cite{Biasotti_C_F_G_L_2009}.
The formulas in that paper may be understood in terms of a sublinear projection, satisfying $\omega \dashv \Omega$, that is defined when $\vect{a}$ is strictly positive:
\[
\omega_\Gamma =
\sup_{\vect{x}, k}
\left[ \frac{ \left( \Gamma(\vect{x}) - \vect{x} \right)_k}{a_k} \right]
\]
\end{remark*}

\begin{proposition}
The metrics defined above are related by $\partial_{\vect{a},\vect{b}} \leq \met_\vect{a}$.
\end{proposition}

\begin{proof}
The translations $(\Omega_\eps)$ on~$\cat{R}^n$ restrict to the standard family of translations on~$\cat{R}$, when $\cat{R}$ is identified with its copy $L(\cat{R}) \subset \cat{R}^n$.
It follows that if $F, G$ are $\Omega_\eps$-interleaved then the interleaving maps restrict to an $\eps$-interleaving of $FL, GL$.
\end{proof}

Another option is to express the relationship between two persistence modules $F,G$ as a more complicated object, namely as a subset of $[0,\infty)^n$. 
Either side of the diagram
\[
\begin{diagram}
\node{[0,\infty)^n}
	\arrow{e,tb}{\Omega}{\mathrm{lax}}
\node{\cat{\Trans}_{\cat{P}}}
	\arrow{e,tb}{\omega}{\mathrm{oplax}}
\node{[0, \infty]^n}
\end{diagram}
\]
will do.

\begin{definition}
Given $\omega: \cat{\Trans}_{\cat{P}} \to [0,\infty]^n$ satisfying $\omega_{\Gamma_1 \Gamma_2} \leq \omega_{\Gamma_1} + \omega_{\Gamma_2}$, we define:
\[
\D(F,G)
=
\left\{ \vect{e} \in [0,\infty)^n
\mid \text{$F,G$ are $(\Gamma, \Kappa)$-interleaved, for $\Gamma, \Kappa$ with $\omega_\Gamma, \omega_\Kappa \leq \vect{e}$} \right\}
\]
Alternatively, $\Omega: [0,\infty)^n \to \cat{\Trans}_\cat{P}$ satisfying $\Omega_{\vect{e}_1 + \vect{e}_2} \geq \Omega_{\vect{e}_1} \Omega_{\vect{e}_2}$, we define:
\[
\D(F,G)
=
\left\{  \vect{e} \in [0,\infty)^n
\mid \text{$F,G$ are $\Omega_\vect{e}$-interleaved} \right\}
\]
Each $\D(F,G)$ is an \emph{up-set} in the sense that if $\vect{e} \in \D(F,G)$ and $\vect{e} \leq \vect{e}'$ then $\vect{e}' \in \D(F,G)$. This is immediate for the $\omega$-definition and follows from monotonicity for the $\Omega$-definition.
\end{definition}

If both maps are given and $\omega \dashv \Omega$, then the two definitions lead to the same set. This follows from the argument used in the proof of Theorem~\ref{thm:equality}.

\begin{example}
The standard adjoint pair on~$\cat{R}^n$ is the family $(\Omega_\vect{e})$ defined by
\[
\Omega_\vect{e}(\vect{x}) = \vect{x} + \vect{e}
\]
for nonnegative $\vect{e} = (e_1, \dots, e_n)$, and the projection $\omega$ defined by
\[
\omega_\Gamma =
\vect{sup}_\vect{x} \left( \Gamma(\vect{x}) - \vect{x} \right)
\]
where the vector supremum is taken component-wise.
It is easy to see that
\[
\met_\vect{a}(F,G)
=
\inf \left( \eps \mid \eps \vect{a} \in \D(F,G) \right)
\]
with this choice.
\end{example}

We record the following properties of~$\D$, which follow easily from the elementary properties of interleaving as well as (for the second item) the sublinearity of~$\omega$ or superlinearity of~$\Omega$.

\begin{proposition}
For any lax monoidal functor $\Omega$ or oplax monoidal functor~$\omega$ as above, the following statements are valid for $E, F, G: \cat{P} \to \cat{D}$ and any functor $H$ with domain~$\cat{D}$.
\begin{vlist}

\item
$\D(HF,HG) \supseteq \D(F,G)$.

\item
$\D(E, G) \supseteq \D(E,F) + \D(F,G)$.
\end{vlist}
(The sum in the second item is to be understood in the sense of Minkowski.)
\qed
\end{proposition}

\begin{example}[refining Example~\ref{ex:multi-d}] Let $X$ be a topological space, let $\vect{f}, \vect{g}: X \to \R^n$, and let $F, G : \cat{R}^n \to \cat{Top}$ denote the corresponding multi-filtrations, defined using lower quadrants:
\[
F(\vect{a}) = \vect{f}^{-1}(Q^\vect{a}),
\quad
G(\vect{a}) = \vect{g}^{-1}(Q^\vect{a}).
\]
Define $\vect{e} = (\eps_1, \dots, \eps_k)$ where the $\eps_k = \norm{f_k-g_k}_\infty$ are component-wise norms. Then we have inclusions
\[
F(\vect{a}) \subseteq G(\vect{a} + \vect{e}),
\quad
G(\vect{a}) \subseteq F(\vect{a} + \vect{e})
\]
for all $\vect{a} \in \R^n$. It follows that $\vect{e}$ belongs to $\D(F,G)$ and hence to $\D(HF,HG)$, when these are defined using the standard adjoint pair.
This refines the original statement that $HF, HG$ are $\eps$-interleaved, since $\vect{e} \leq \eps \vect{1}$ when $\norm{\vect{f}-\vect{g}}_\infty \leq \eps$.
\end{example}

We finish Section~\ref{sec:monoids} by remarking that the $n=1$ case is special: an up-set of $[0,\infty)$ is almost completely determined by its infimum. By setting $\met(F,G) = \inf( \D(F,G))$ we return to the metric theory developed in earlier sections of this paper.

\section*{Closing Remarks}

Why category theory? The level of abstraction it provides has a number of advantages. It allows a uniform treatment of many flavors of persistence. We can give simpler common proofs to basic persistence results that have been proved individually in different settings.

Much of the standard theory can be developed for arbitrary functors into an arbitrary target category, provided that the indexing preordered set is equipped with 
a sublinear projection or superlinear family. 
In one direction, we can consider persistence for functions with values in an arbitrary metric space in place of the usual $\R$ or~$\R^n$. This expands the territory of topological data analysis.
In another direction, we can replace homology with other functors from algebraic topology, some of which may be of use in computations, such as rational homotopy or homology of the loop space. We can build merge trees~\cite{Morozov_B_W_2013} using the connected components functor~$\pi_0$, for instance.

The present work suggests a philosophy for developing persistence in new settings.
We expect theories of persistence to have two ingredients. The first ingredient is the construction of a generalized persistence module from data. Our methods suggest ways of doing this, metrics, and easy proofs of stability.
The second ingredient is a discrete invariant on the chosen class of modules. This is because generalized persistence modules typically carry a lot of information, and their interleaving metric may not be effectively computable. A good invariant will be easy to compute, and will have a metric that is easy to compute. The invariant will be stable with respect to variation in the module, but it will vary enough to be a useful discriminator.

This division into two stages is seen very clearly in classical persistence: from data one constructs a persistence module, and from the module one constructs the barcode or persistence diagram. Historically, the original algorithm of~\cite{Edelsbrunner_L_Z_2000} fused these two steps. Nowadays, the intermediate theoretical object, the persistence module, is indispensable for a thorough understanding of persistence~\cite{Zomorodian_Carlsson_2005}.

The two stages correspond also to the distinction between soft and hard stability theorems proposed in Section~\ref{subsec:stability}.
We believe it important to separate the aspects of the theory arising from general categorical considerations from the aspects that require arguments specific to the situation at hand.

\subsection*{Acknowledgements}

The first author gratefully acknowledges the support of AFOSR grant
FA9550-13-1-0115.  
The second author thanks his home institution, Pomona College, for a sabbatical leave of absence in 2013--14. The sabbatical was partially supported by the Simons Foundation (grant \#267571); and hosted by the Institute for Mathematics and its Applications, University of Minnesota, with funds provided by the National Science Foundation.


\end{document}